\documentclass[11pt, notitlepage]{article}
\usepackage{amssymb,amsmath,comment}
\usepackage{amsthm}
\catcode`\@=11 \@addtoreset{equation}{section}
\def\thesection{\arabic{section}}

\def\theequation{\thesection.\arabic{equation}}
\catcode`\@=12
\usepackage{colortbl}
\usepackage{hyperref}
\usepackage[mathscr]{eucal}
\usepackage{epsf}
\usepackage{esint}
\usepackage{a4wide}
\def\R{\mathbb{R}}

\newcommand{\e}{\varepsilon}

\newcommand{\al} {\alpha}
\newcommand{\ba} {\beta}
\newcommand{\de} {\delta}

\newcommand{\Om} {\Omega}
\newcommand{\ra} {\rightarrow}

\newcommand{\De} {\Delta}
\newcommand{\la} {\lambda}
\newcommand{\La} {\Lambda}
\newcommand{\noi} {\noindent}
\newcommand{\na} {\nabla}

\newcommand{\mb} {\mathbb}
\newcommand{\mc} {\mathcal}

\newcommand{\ld} {\langle}
\newcommand{\rd} {\rangle}
\newcommand{\I}{\int\limits_}

\usepackage[all]{xy}
\catcode`\@=11
\def\theequation{\@arabic{\c@section}.\@arabic{\c@equation}}
\catcode`\@=12

\newtheorem{Theorem}{Theorem}[section]
\newtheorem{Lemma}[Theorem]{Lemma}
\newtheorem{prop}[Theorem]{Proposition}

\newtheorem{Remark}[Theorem]{Remark}
\newtheorem{Definition}[Theorem]{Definition}

\begin{document}

{\vspace{0.01in}}

\title{Symmetry, existence and regularity results for a class of mixed local-nonlocal semilinear singular elliptic problem via variational characterization}

\author{ {\bf Gurdev Chand Anthal\footnote{	 Indian Institute of Science Education and Research Berhampur, Berhampur, Ganjam, Odisha, India-760010. email: Gurdevanthal92@gmail.com}\; and Prashanta Garain\footnote{Indian Institute of Science Education and Research Berhampur, Berhampur, Ganjam, Odisha, India-760010.
					email: pgarain92@gmail.com}  }}

\maketitle

\begin{abstract}\noindent
In this article, we present the symmetry of weak solutions to a mixed local-nonlocal singular problem. We also establish results related to the existence, nonexistence, and regularity of weak solutions to a mixed local-nonlocal singular jumping problem. A crucial element in proving our main results is the variational characterization of the solutions, which also reveals the decomposition property. This decomposition property, together with comparison principles and the moving plane method, yields the symmetry result. Additionally, we utilize nonsmooth critical point theory alongside the variational characterization to analyze the jumping problem.

\end{abstract}

\maketitle

\noi {Keywords: Mixed local-nonlocal singular problem, variational characterization, decomposition, comparison principles, moving plane method, symmetry, jumping problem, existence, regularity.}

\noi{\textit{2020 Mathematics Subject Classification: 35M10, 35J75, 35B06, 35B25, 35B65, 35R11.}

\bigskip

\tableofcontents

\section{Introduction}\label{Intro}
We consider the following class of mixed local-nonlocal semilinear elliptic equation with singular nonlinearity
$$
(P_{\gamma,w}): \quad \quad \quad \quad \mathcal{M}u:=-\Delta u+(-\Delta)^s u=u^{-\gamma}+w\text{ in }\Om,\quad u>0\text{ in }\Om,\quad u=0\text{ in }\mb R^n\setminus\Om, 
$$
where $\Om\subset\mb R^n$, $n\geq 3$ is a a bounded $C^{1,1}$ domain and $\gamma>0$. Here $\Delta$ is the classical Laplace operator and $(-\Delta)^s$ is the fractional Laplace operator defined by
$$
(-\Delta)^s u(x)=\text{P.V.}\int_{\mb R^n}\frac{(u(x)-u(y))}{|x-y|^{n+2s}}\,dy,\quad 0<s<1
$$
where P.V. denotes the principal value. The operator $\mathcal{M}:=-\Delta+(-\Delta)^s$ is referred to as the mixed local-nonlocal operator, see \cite{Valap} for its physical applications.\\
In the first part of this article, we establish symmetry of weak solutions of the problem $(P_{\gamma,w})$, where $w$ takes the form $(\mc H_1)$ given by
\begin{itemize}
\item[$(\mc H_1)$] $w=\wp(u)$, where $\wp$ satisfies the hypothesis $(\mc A)$ below:
	 	\item[$(\mc A)$] $\wp(\cdot)$~\text{is locally Lipschitz continuous, non-decreasing},\,$\wp(t) >0~\text{for}~t>0~\text{and}~\wp(0) \geq 0$.
	\end{itemize}
The second part of this article is devoted to study the existence, non-existence and regularity of weak solutions of the problem $(P_{\gamma,w})$, where $w$ takes the form $(\mc H_2)$ given by
     \begin{itemize}
\item[$(\mc H_2)$] $w=h(x,u)-\la e_1$, where $\la\in\R$ and $h:\Om\times\mb R\to\mb R$ is a Carath\'eodory function satisfying the conditions $(h_1)$ and $(h_2)$ below:
	 			\item [$(h_1)$] there exists a constant $C>0$ such that
	 			\begin{align*}
	 				|h(x,s)| \leq C(1+|s|)\quad ~\text{for}~x \in \Om~\text{and every}~s \in \mb R,
	 			\end{align*}
	 			\item [$(h_2)$] there exists $\alpha \in \mb R$ such that
	 			\begin{align*}
	 				\lim_{s \ra +\infty} \frac{h(x,s)}{s} =\alpha\quad ~\text{for}~x\in \Om.
	 			\end{align*}
	 		\end{itemize}
             Here $e_1\in H_0^{1}(\Omega)\cap L^\infty(\Omega)\,(e_1>0\text{ in }\Omega)$ is the first eigenfunction of $\mathcal{M}$ in $\Omega$, with the associated first eigenvalue $\lambda_1$ (refer to \cite[Theorem B.1]{SVWZ}, \cite[Theorems 2.3 and 2.4]{GUjms}), which satisfies
             \begin{equation}\label{mevp}
\mc M e_1=\la_1 e_1\text{ in }\Om,\quad e_1>0\text{ in }\Om,\quad e_1=0\text{ in }\mb R^n\setminus\Om.
             \end{equation}
             The singularity of the problem $(P_{\gamma,w})$ is captured by the positivity of the singular exponent $\gamma>0$, which leads to the blow-up behavior of the nonlinearity on the right-hand side of $(P_{\gamma,w})$. Singular elliptic problems have been thoroughly investigated over the past three decades, in both the local case \cite{Arcoyana, BocOr, CRT, Radu, LazerMc, SMathZ} and the nonlocal case \cite{Adijde, Peralopen, CaninoSc, GMuk, YGcvpde}, along with the references mentioned therein.
             
            Regarding symmetry results, in the local case, we highlight the work \cite{CGS}, where the authors established the symmetry of positive classical solutions to the following singular Laplace equation:
             \begin{equation}\label{Canjde}
                 -\Delta u=u^{-\gamma}+\wp(u)\text{ in }\Om,\quad u>0\text{ in }\Om,\quad u=0\text{ on }\partial\Om,
             \end{equation}
             where $\Om\subset\mb R^n$ is a bounded smooth, strictly convex, symmetric domain and $\wp$ satisfies the hypothesis $(\mc A)$ mentioned earlier. Symmetry results for a more general version of equation \eqref{Canjde} are explored in \cite{BalJac, Caninona, Sjfa, Tans} and the references therein. Furthermore, in the nonlocal case, the authors in \cite{Jacsym} proved symmetry result for the following singular fractional Laplace equation:
             \begin{equation}\label{Jacetal}
                 (-\Delta)^s u=u^{-\gamma}+\wp(u)\text{ in }B_r(0),\quad u>0\text{ in }B_r(0),\quad u=0\text{ in }\mb R^n\setminus B_r(0),
             \end{equation}
             where $B_r(0)\subset\mb R^n$ is the ball of radius $r$ centered at the origin $0=(0,0,\ldots,0)\in\mb R^n$.

             Related to the jumping problem, in the local case, the singular Laplace equation
              \begin{equation}\label{Lapjum}
             -\Delta u=f(x)u^{-\gamma}+h(x,u)-t\phi_1\text{ in }\Om,\quad u>0\text{ in }\Om,\quad u=0\text{ on }\partial\Om,
             \end{equation}
             where $\phi_1$ is the first eigenfunction of $-\Delta$ in a bounded $C^{1,1}$ domain $\Om\subset\mb R^n$ with the Dirichlet boundary condition is studied in \cite{Caninojumping} where $f\equiv 1$ in $\Om$, and $h:\Om\times\mb R\to\mb R$ is a Careth\'eodory function satisfying $(h_1)$ and $(h_2)$. Furthermore, equation \eqref{Lapjum} is studied for general $f$ and $h$ in \cite{Canino14}. In addition, the following nonlocal jumping problem
              \begin{equation}\label{fLapjum}
             (-\Delta)^s u=u^{-\gamma}+g(x,u)-t\psi_1\text{ in }\Om,\quad u>0\text{ in }\Om,\quad u=0\text{ on }\mb R^n\setminus\Om,
             \end{equation}
             where $0<s<1,\,n>2s,\,t\in\mb R,\,\gamma>0$, and $\psi_1$ is the first eigenfunction of $(-\Delta)^s$ in a bounded smooth domain $\Om\subset\mb R^n$ with the Dirichlet boundary condition, is studied in \cite{Caninodcds}. Here $g:\Om\times\mb R\to\mb R$ is some Carath\'eodory function satisfying certain growth conditions.

In the mixed local-nonlocal case, non-singular problems have been investigated in \cite{Chen, Biagicpde, Min, Foondun,  GK,  GLcvpde, Var23}, as well as in the references therein. Recently, the study of mixed local-nonlocal singular problems has garnered significant attention. In this regard, we refer to \cite{AroRad}, where the authors explored the following purely singular mixed local-nonlocal problem, focusing on existence and regularity results:
             \begin{align}\label{pseqn}
             \begin{cases}
					\mc M u = f(x){u^{-\gamma}}~&\text{in}~\Om,\\
                    u>0~&\text{in}~\Om,\\
					u =0~&\text{in}~\mb R^n \setminus \Om,
				\end{cases}
			\end{align}
            where $\Om\subset\mathbb{R}^n$ is a bounded $C^{1,1}$ domain. Here $\gamma>0$ and $f:\Om\to\mathbb{R}^+$ either belong to $L^r(\Om)\setminus\{0\}$ for some $r\geq 1$, or exhibits growth corresponding to negative powers of the distance function near the boundary. Additionally, in \cite{GUna}, the authors examined the quasilinear version of the problem \eqref{pseqn}, proving existence, uniqueness, and symmetry results for any $\gamma>0$, under the assumption that $f\in L^r(\Omega)\setminus\{0\}$ is a non-negative function for some $r\geq 1$. We also refer to \cite{Gjms, LzMc} for studies of purely singular mixed local-nonlocal problems. Furthermore, purely singular mixed local-nonlocal problems with variable singular exponents have been investigated in \cite{Biroud, GKK} and the references therein.
            
            In the purturbed singular mixed local-nonlocal case, consider the following mixed local-nonlocal problem: 
            \begin{equation}\label{mspp}
            \mc Mu=\la u^{-\gamma}+u^q\text{ in }\Om,\quad u>0\text{ in }\Om,\quad u=0\text{ in }\mb R^n\setminus\Om,
            \end{equation}
            where $\Om\subset\mb R^n$ is a bounded smooth domain, $\la>0$, $\gamma\in(0,1)$, and $q\in (1,2^*-1]$ with $2^*=\frac{2n}{n-2}$ if $n>2$. Multiplicity results for a certain range of $\la>0$ are obtained in the subcritical case $q\in(1,2^*)$ in \cite{Garainjga}, and in the critical case $q=2^*$ in \cite{BV} respectively. We remark that, the case of any $\gamma>0$ is addressed in \cite{BalDas}, where the authors proved multiplicity results for the problem \eqref{mspp} for a certain range of $\la>0$, assuming $\Om$ to be a strictly convex bounded domain in $\mb R^n$ with $s\in(0,\frac{1}{2})$ and $q\in(1,2^*-1)$. Moreover, when $\gamma\in(0,1)$, multiplicity results for the associated quasilinear problem of \eqref{mspp} are established for some range of $\la>0$ in the subcritical case in \cite{BalDas}, for any bounded $C^1$ domain in $\mb R^n$. Recently, mixed local-nonlocal singular problems with measure data have also been studied, with results in \cite{BG24} for a constant singular exponent and in \cite{BGccm} for a variable singular exponent.
            
            To the best of our knowledge, symmetry results for perturbed singular mixed local-nonlocal problems are not yet known, and the singular jumping problem in the mixed local-nonlocal case has not been studied. The primary goal of this article is to address these gaps.

Due to the singularity, one of the main challenges we encounter is that, in general, solutions to mixed local-nonlocal singular problems belong to the local Sobolev space, as discussed in \cite{AroRad} and the references therein. As a result, the standard variational method cannot be directly applied in our setting. We address this difficulty by establishing a variational characterization of the solutions to the problem $(P_{\gamma,w})$ assuming $w\in H^{-1}(\Om)$ (see Theorem \ref{t1.1}). To achieve this, we adapt the approaches from \cite{CD, CMST} to the mixed local-nonlocal singular problem. Furthermore, we demonstrate that the solutions to this variational inequality are minimizers of a suitable functional (see Theorem \ref{t3.2}). Recently, minimax principles for hemivariational-variational inequalities have been studied in \cite{Rep24, Rep22} and the references therein. Additionally, we mention the recent works \cite{Lei1, Lei2} on singular elliptic problems, where variational characterization in terms of lower critical points \cite{Marco83, Marcoampa} is used. 

To establish the symmetry result, we primarily employ the moving plane technique \cite{BVDDS}. However, it is important to note that because the nonlinearity $u^{-\gamma}+\wp(u)$ is not Lipschitz at the origin, this technique cannot be applied directly. To overcome this, we adopt the approach introduced in \cite{CGS}, which combines decomposition and the moving plane method.  Specifically, by Theorem \ref{t1.2} (which follows from Theorem \ref{t1.1}), every weak solution $u$ of $(P_{\gamma,w})$ with $\gamma>0$ and $w$ of the form $(\mc H_1)$ can be decomposed as $u=u_0+z$, where $u_0$ is in the local Sobolev space, which is a solution of a purely singular mixed local-nonlocal problem and $z$ is a Sobolev function taking zero boundary value. Thus, to prove the symmetry of $u$, it suffices to establish the symmetry of $u_0$ and of $z$. This will primarily be achieved using the moving plane technique. However, in order to establish the symmetry of $z$, we first prove some comparison principles for $z$, which will allow us to apply the moving plane technique.

             To study the jumping problem $(P_{\gamma,w})$ when $w$ takes the form $(\mc H_2)$, we mainly apply the nonsmooth critical point theory as developed in \cite{S} and combine the approaches from \cite{Caninojumping, Caninodcds} to address the mixed local-nonlocal setting. We demonstrate that these critical points are indeed the weak solutions by showing that they satisfy a specific variational inequality. To this end, the variational characterization results in Theorem \ref{t3.2} and \ref{t1.1} are crucial.
\subsection{Notations}
We will use the following notations throughout the remainder of the article, unless stated otherwise.
	\begin{enumerate}
		\item $d\nu =|x-y|^{-n-2s}\,dxdy$,\quad $\mathcal{B}(u)(x,y)=u(x)-u(y)$.
        
        \item $\gamma$ will denote a positive constant and $0<s<1$. 

  \item $\Omega$ will denote a bounded $C^{1,1}$ domain in $\mathbb{R}^n$, $n\geq 3$.

  \item For a given space $W$ and a given subset $S$ of $\mb R^n$, we denote by $W_c(S)$ to mean the set of functions in $W(S)$ that have compact support within $S$.

  \item For bounded subsets $U, V$ of $\mathbb{R}^n$, we denote $V\Subset U$, to mean that $V\subset\overline{V}\subset U$.

  \item For $u\in H_0^{1}(\Om)$, we use the notation $\|u\|$ to mean
  $$
  \|u\|:=\|u\|_{H_0^{1}(\Om)}=\left(\I\Om |\na u|^2 dx +\iint_{\mb R^{2n}}{|\mc B u(x,y)|^2}\,d\nu\right)^\frac12.
  $$

  \item We denote $(\mc P_{\gamma,0})$ to mean the problem $(\mc P_{\gamma,w})$ after putting $w=0$.
  
  \item For a given real-valued function $f$ defined on a set $S$ of $\mb R^n$ and given constants $c,d$, we write $c\leq f\leq d$ on $S$ to mean that $c\leq f\leq d$ on $S$ almost everywhere in $S$.
       \end{enumerate}

	\section{Functional setting, auxiliary results and main results}
	\subsection{Functional setting}
In this subsection, we outline the functional setting.
     {The Sobolev space $H^1(\Omega)$ consists of functions $u:\Omega\to\mathbb{R}$ that belong to $L^2(\Omega)$, for which the partial derivatives $\frac{\partial u}{\partial x_i}$ (for $1\leq i\leq n$) exist in the weak sense and are elements of $L^2(\Omega)$. The space $H^1(\Omega)$ is a Banach space (see \cite{Evans}) with the norm defined as:
\begin{align*}
\|u\|_{H^1(\Om)} =\Big(\|u\|_{L^2(\Om)}^2 +\|\na u\|_{L^2(\Om)}^2\Big)^\frac{1}{2}
\end{align*}
where $\na u=\Big(\frac{\partial u}{\partial x_1},\ldots,\frac{\partial u}{\partial x_n}\Big)$. The space $H^1_{\mathrm{loc}}(\Omega)$ is defined as:
  \begin{align*}
	 	H^1_{\text{loc}}(\Om)=\{u: \Om \ra \mb R: u \in H^1(K)~\text{for all}~ K \Subset \Om\}.
	 \end{align*}
  The space $H^1(\mathbb{R}^n)$ is defined analogously. {Occasionally, we may also require the higher-order Sobolev space $W^{2,p}(\Om)$ with $1\leq p<\infty$ as well, which is the standard Sobolev space. For a detailed definition and more information, we refer to \cite{Evans}.}
  To address mixed problems, we define the space
  $$
  H_0^1(\Om)=\{u\in H^1(\mathbb{R}^n):u=0\text{ in }\mathbb{R}^n\setminus\Omega\}.
  $$
	 Next, we recall the concept of fractional order Sobolev spaces from \cite{DPV}. For $s \in (0,1)$, the fractional Sobolev space $H^s(\Om)$ is defined as
	 	\begin{align*}
	 		H^s(\Om) =\left\{u \in L^2 (\Om): \frac{|u(x)-u(y)|}{|x-y|^{\frac{n}{2}+s}} \in L^2(\Om \times \Om)\right\}
	 	\end{align*}
	 	and it is endowed with the norm
	 	\begin{align*}
	 		\|u\|_{H^s(\Om)} =\Big(\|u\|_{L^2(\Om)}^2 +[u]_{s,\Om}^2\Big)^\frac{1}{2},
	 	\end{align*}
	 	where
	 	\begin{equation*}
	 		[u]_{s,\Om} =\left( \I{\Om}\I{\Om}{|\mc B(u)(x,y)|^2}\,d\nu\right)^\frac12.
	 	\end{equation*}
        Similarly, we define
        \begin{equation*}
	 		[u]_{s,\mb R^n} =\left( \iint_{\mb R^{2n}}{|\mc B(u)(x,y)|^2}\,d\nu\right)^\frac12.
	 	\end{equation*}
   The space $H^s(\mathbb{R}^n)$ is defined in a similar manner. The following result demonstrates that the Sobolev space $H^1(\Om)$ is continuously embedded within the fractional Sobolev space, as shown in \cite[Proposition 2.2]{DPV}.
	 \begin{Lemma}\label{l2.2}
	 	There exists a constant $C =C(n,p,s)>0$ such that
	 	\begin{align*}
	 		\|u \|_{H^s(\Om)} \leq C \|u\|_{H^1(\Om)},\quad \forall u \in H^1(\Om).
	 	\end{align*}
	 \end{Lemma}
	 Next, we present the following result from \cite[Lemma 2.1]{BSM}.
	 \begin{Lemma} \label{l2.3}
	 	There exists a constant $C= C(n,p,s,\Om)$ such that
	 	\begin{align}\label{e2.1}
	 		\iint_{\mb R^{2n}}{| \mc B(u)(x,y)|^2}\,d\nu \leq C \I\Om|\nabla u(x)|^2dx,\quad \forall u \in H_0^1(\Om).
	 	\end{align}
	 \end{Lemma}
	 \begin{Remark}\label{rm2.4}
	 	By combining \eqref{e2.1} with the Poincaré inequality, we can observe that the following norms on the space $H_0^1(\Om)$, defined for $u\in H_0^{1}(\Om)$, are equivalent:
	 	\begin{equation*}\label{norm}
	 		\|u\|: =\left(\I\Om |\na u|^2 dx +\iint_{\mb R^{2n}}{|\mc B u(x,y)|^2}\,d\nu\right)^\frac12,
	 	\end{equation*}
        and 
	 	\begin{align*}
	 		\|u\|_2: = \|\na u\|_{L^2(\Om)}.
	 	\end{align*}
	 \end{Remark}
For more information on the space $H_0^{1}(\Omega)$, refer to \cite{Vecchihong, VecchiBO, Vecchihenon} and the references therein. The dual space of $H^{1}_{0}(\Om)$ is denoted by $H^{-1}(\Om)$.

As noted in \cite{AroRad, GUna}, solutions to singular problems are generally not elements of $H_0^{1}(\Omega)$ for large $\gamma>0$. Therefore, boundary values are understood in the following sense:
	 \begin{Definition}\label{d2.5}
	 	We say that $u \leq 0$ on $\partial \Om$ if $u =0$ in $\mb R^n \setminus \Om$ and for every $\e >0$, we have 
	 	\begin{align*}
	 		(u-\e)^+ \in H_0^1(\Om).
	 	\end{align*}
	 	We will say $u =0$ on $\partial \Om$ if $u$ in non-negative and $u \leq 0$ on $\partial \Om$.
	 \end{Definition}
	 Before presenting the weak formulation of $(\mc P_{\gamma,w})$, we state the following proposition, the proof of which follows similarly to that of \cite[Proposition 2.3]{CMST} by using Lemmas \ref{l2.2} and \ref{l2.3}.
	 \begin{prop}\label{p2.6}
	 	Let $u \in H^1_{\text{loc}}(\Om) \cap L^1(\Om)$ and $u =0$ for  $x \in \mb R^n \setminus \Om$. Then for every $ \varphi \in C_c^\infty (\Om)$, we have
	 	\begin{align*}
	 	\iint_{\mb R^{2n}} \mathcal{B}(u)(x,y)\mathcal{B}(\varphi)(x,y)\,d\nu < \infty.
	 	\end{align*} 	
	 \end{prop}
	 In view of Proposition \ref{p2.6}, we introduce the following definition of a weak solution to the problem $(\mc P_{\gamma,w})$:
	 \begin{Definition}\label{dws}
	 	Let $w\in H^{-1}(\Om)$. A function $u \in H^1_{\text{loc}}(\Om) \cap L^1(\Om)$ is said to be a weak solution to the problem $(\mc P_{\gamma,w})$ if:
	 		\begin{enumerate}
	 		\item $u>0$ in $\Om$, $ u=0$ on $\partial \Om$ in the sense of Definition \ref{d2.5} and $u^{-\gamma} \in L^1_{\text{loc}}(\Om)$.
	 		\item For every $\varphi \in C_c^\infty(\Om)$, we have
	 		\begin{align*}
	 			\I{\Om} \na u\na \varphi dx + \iint_{\mb R^{2n}}  \mathcal{B}(u)(x,y)\mathcal{B}(\varphi)(x,y)\,d\nu =\I{\Om} u^{-\gamma} \varphi dx +\ld w,\varphi\rd. 
	 		\end{align*}
	 	\end{enumerate}
	 	\end{Definition}
	 	Next we define the notion of weak subsolutions and weak supersolutions of the problem $(\mc P_{\gamma,w})$.
	 	\begin{Definition} (Weak supersolution)
	 		Let $w\in H^{-1}(\Om)$. A function $v \in H^1_{\text{loc}}(\Om) \cap L^1(\Om)$ is said to be a weak supersolution to $(\mc P_{\gamma,w})$, if
	 		\begin{enumerate}
	 			\item $v>0$ in $\Om$, $ v=0$ in $\mb R^n \setminus \Om$ and $v^{-\gamma} \in L^1_{\text{loc}}(\Om)$. 
	 		
	 		\item  For every $\varphi \in C_c^\infty(\Om)$ with $\varphi \geq 0$, we have
	 		\begin{align*}
	 				\I{\Om} \na v\na \varphi dx + \iint_{\mb R^{2n}} \mathcal{B}(v)(x,y)\mathcal{B}(\varphi)(x,y)\,d\nu \geq \I{\Om} v^{-\gamma} \varphi dx +\ld w,\varphi\rd.
	 		\end{align*}
	 		\end{enumerate}
	 		\end{Definition}
	 		\begin{Definition} (Weak subsolution)
	 			Let $w\in H^{-1}(\Om)$. A function $v \in H^1_{\text{loc}}(\Om) \cap L^1(\Om)$ is said to be a weak subsolution to $(\mc P_{\gamma,w})$, if
	 			\begin{enumerate}
	 				\item $v>0$ in $\Om$, $ v=0$ in $\mb R^n \setminus \Om$ and $v^{-\gamma} \in L^1_{\text{loc}}(\Om)$. 
	 	\item For every $\varphi \in C_c^\infty(\Om)$ with $0 \leq \varphi \leq v$, we have
	 		\begin{align*}
	 			\I{\Om} \na v\na \varphi dx + \iint_{\mb R^{2n}} \mathcal{B}(v)(x,y)\mathcal{B}(\varphi)(x,y)\,d\nu \leq  \I{\Om} v^{-\gamma} \varphi dx +\ld w,\varphi\rd.
	 		\end{align*}
	 	\end{enumerate}
	 	\end{Definition}

        \subsection{Auxiliary results} 
        {In this subsection, we present some auxiliary results that are essential for deriving our main results. The first of these is a comparison result, which can be obtained by reasoning in a manner similar to the proof of \cite[Lemma 4.5]{GUna}.}
	 	\begin{Lemma}\label{tcp}
	 		Let $\gamma >0$ and $w \in H^{-1}(\Om)$. Suppose $v\in H^{1}_{\mathrm{loc}}(\Om)\cap L^1(\Om)$ is a weak subsolution to $(\mc P_{\gamma,w})$ such that $v \leq 0$ on $\partial \Om$, and let $z\in H^{1}_{\mathrm{loc}}(\Om)\cap L^1(\Om)$ be a weak supersolution to $(\mc P_{\gamma,w})$. Then, it follows that $v \leq z$ in $\Om$.
	 	\end{Lemma}

Next, we present the following result from \cite[Lemma 4]{CGS}.
	\begin{Lemma}\label{l4.1}
		Let $\gamma>0$ and consider the function $\Re_\gamma:U\to\mathbb{R}$ defined by
		\begin{align*}
\Re_\gamma(x,y,j,m):=x^\gamma(x+y)^\gamma(j+m)^\gamma+ x^\gamma j^\gamma(j+m)^\gamma-j^\gamma(x+y)^\gamma(j+m)^\gamma-x^\gamma j^\gamma(x+y)^\gamma,
		\end{align*}
		where the domain $U \subset \mb R^4$ is defined by
		\begin{align*}
			U:=\{(x,y,j,m): 0 \leq x \leq j,\, 0\leq m \leq y\}.
		\end{align*}
		Then it follows that $\Re_\gamma \leq 0$ in $U$. 
	\end{Lemma}

	 	We also recall an extension of the celebrated Mountain pass theorem (see \cite{S}) of Ambrosetti and Rabinowitz {as stated in Theorem \ref{TMPT} below,} which will be used to study the jumping problem.
	 	\begin{Definition}\label{dcp}
	 		Let $V$ be a real Banach space, and suppose $ \mc J =\mc F +\mc K$, where $\mc F: V \ra (-\infty,+\infty]$ is convex, proper (i.e. $\mc J \not \equiv +\infty$) and lower semicontinuous functional, and $\mc K: V \ra \mb R$ is a functional of class $C^1$. We say that $u \in V$ is a critical point of $\mc J$, if
	 		\begin{align*}
	 			\mc F(v) \geq \mc F(u)-\ld \mc K^\prime(u),v-u\rd, ~\forall v \in V.
	 		\end{align*}
	 	\end{Definition}
	 	\begin{Definition}
	 		As in Definition \ref{dcp}, we say that $\mc J$ satisfies the Palais-Smale (PS) condition if, for every sequence $\{u_k\}_{k \in \mb N}$ in $V$ and $\{\omega_k\}_{k \in \mb N}$ in $V^\ast$ such that $\sup_{k \in \mb N} |\mc J(u_k)| <+\infty$, $\omega_k \ra 0$, and
	 		\begin{align*}
	 			\mc F(v) \geq \mc F(u_k)-\ld \mc K^\prime(u_k),v-u_k\rd +\ld \omega_k,v-u_k\rd,~\forall v \in V,
	 		\end{align*}
	 		the sequence $\{u_k\}_{k \in \in \mb N}$ has a convergent subsequence in $V$.
	 	\end{Definition}
	 	\begin{Theorem}\label{TMPT}
	 		As in Definition \ref{dcp}, assume that $\mc J$ satisfies the $(PS)$ condition, and that there exist $r >0$ and $\sigma > \mc J(0)$ such that 
	 		\begin{align*}
	 			&\mc J(u) \geq \sigma,~\forall u \in V~\text{with}~\|u\|=r,\text{ and }\\
	 			&	\mc J(u_1) \leq \mc J(0),~\text{for some}~u_1 \in V~\text{with}~\|u_1\| >r.
	 		\end{align*}
	 		Then there exists a critical point $v$ of $\mc J$ with $\mc J(v) \geq \sigma$.
	 	\end{Theorem}

	 \subsection{Main results}
     We are now ready to present our main results, which are stated as follows: The first major result establishes the symmetry of weak solutions to the problem $(\mc P_{\gamma,w})$ when $w$ is of the form $(\mc H_1)$.
	 \begin{Theorem}\label{t1.3}(\textbf{Symmetry})
	 	Let $\gamma >0$ and $w$ be of the form $(\mc H_1)$ specifically $w=\wp(u)$, where $\wp(u)$ satisfies the hypothesis $(\mc A)$. Assume that $\Om \subset \mb R^n$ is a bounded smooth domain. Furthermore, {suppose $\Om$ is {strictly convex} with respect to the $x_1$-direction and symmetric with respect to the hyperplane $\{x_1=0\}$.} Then every weak solution $u \in H^1_{\mathrm{loc}}(\Om)\cap L^1(\Om)\cap C(\overline{\Om}) $ of $(\mc P_{\gamma,w})$ is symmetric with respect to $\{x_1 =0\}$. Moreover, if $\Om$ is a ball, then $u$ is radially symmetric.
	 \end{Theorem}
     Our next two main results, presented below, address the jumping problem $(\mc P_{\gamma,w})$, which occurs when $w$ takes the form $(\mc H_2)$.
	 		\begin{Theorem}\label{t1.4}(\textbf{Multiplicity and regularity for large $\lambda$})
	 			Let $\gamma>0$, $\alpha >\la_1$ and $w$ be of the form $(\mc H_2)$. {Assume that $\Om \subset \mb R^n$ is a bounded $C^{1,1}$ domain.} Then there exists $\overline{\la} \in \mb R$ such that for every $\la >\overline{\la}$, the problem $(\mc P_{\gamma,w})$ has at least two distinct weak solutions in $H^1_{\mathrm{loc}}(\Om)\cap L^1(\Om)$. Moreover, these weak solutions belong to $C(\overline{\Om})\cap\left(\cap_{1\leq p<\infty}W^{2,p}_{\text{loc}}(\Om)\right)$ if $s \in (1,1/2]$, and to $C(\overline{\Om})\cap\left(\cap_{1\leq p<n/(2s-1)}W^{2,p}_{\text{loc}}(\Om)\right)$ if $s \in (1/2,1)$.
	 		\end{Theorem}
	 		\begin{Theorem}\label{t1.5}(\textbf{Nonexistence for small $\lambda$})
	 			Let $\gamma>0$, $\alpha>\la_1$, and $w$ be of the form $(\mc H_2)$. {Assume that $\Om \subset \mb R^n$ is a bounded $C^{1,1}$ domain.} Then there exists $\underline{\la} \in \mb R$ such that for every $\la <\underline{\la}$, the problem $(\mc P_{\gamma,w})$ has no weak solution in $H^1_{\text{loc}}(\Om)\cap L^1(\Om)$.
	 		\end{Theorem}

    \textbf{Organization of the Paper:} Sections 3 and 4 are dedicated to the the proof of symmetry and the proof of the main results related to jumping problems, respectively. Lastly, in the Appendix Section 5, we present the variational characterization and the decomposition result.
    
\section{Symmetry result}\label{s4}
This section is dedicated to proving the symmetry result in Theorem \ref{t1.3}, following the approach from \cite{CGS}. Throughout this section, unless stated otherwise, we assume that $w$ takes the form $(\mc H_1)$, i.e., $w=\wp(u)$, where $\wp(u)$ satisfies the hypothesis $(\mc A)$, and $\Om$ represents a bounded smooth domain in $\mb R^n$.

By Theorem \ref{t1.2}, every weak solution $u$ of the problem $(\mc P_{\gamma,w})$ can be decomposed as:
\begin{equation}\label{deco}
u=u_0+z,\quad z\in H_0^{1}(\Om),
\end{equation}
where $\Om\subset\mb R^n$ is a bounded $C^{1,1}$ domain. Here $u_0\in H^1_{\mathrm{loc}}(\Om)\cap L^1(\Om)\cap C(\overline{\Om})$ is the unique weak solution of the problem $(\mc P_{\gamma,0})$ as given by Proposition \ref{p3.1} below. Therefore, to establish the symmetry result of $u$, it suffices to demonstrate the symmetry of $u_0$ and $z$. This will be accomplished in the following steps.

To apply the moving plane technique, we fix some notations that will be used throughout the rest of this section. We denote by $x=(x_1,x_2,\ldots,x_n)$. Without loss of generality, we may assume that
\begin{align*}
	\inf_{x\in \Om} x_1 =-1.
\end{align*}
    For $\la \in (-1,1)$, we define:
    $$
    \mc G_\la:= \{x =(x_1, x_2, \cdots,x_n) \in \mb R^n: x_1 =\la\},
    $$
    and 
	\begin{align*}
		\Sigma_\la := \begin{cases}
			\{x\ \in \mb R^n : x_1 < \la\}~&\text{if}~\la \leq 0,\\
			\{x\ \in \mb R^n : x_1 > \la\}~&\text{if}~\la > 0.
		\end{cases}
	\end{align*}
	Further, for $\la \in (-1,1)$, we define $\Om _\la = \Om \cap \Sigma_\la$, $R_\la(x)=x_\la:= \{2\la -x_1, x_2,\cdots, x_n\}$ be the reflection of the point $x$ about $\mc G_\la$ and $u_\la (x) =u(x_\la)$. Note that $R_\la(\Om_\la)$ may not be contained in $\Om$. So when $\la >-1$, since $\Om_\la$ is nonempty, we set
	\begin{align*}
		\Lambda^\ast =\{\la: R_{\tilde\la}(\Om_{\tilde\la})\subset \Om~\text{for any} -1<\tilde{\la}\leq \la\},
	\end{align*}
	and we define
	\begin{align*}
		\la^\ast =\sup \Lambda^\ast.
	\end{align*}
    \subsection{Properties of $u_0$}
    \subsubsection{Existence and regularity of $u_0$}
   First, we establish the existence and regularity results for weak solutions of the purely singular problem $(\mc P_{\gamma,0})$ for any $\gamma>0$.
		\begin{prop}\label{p3.1}
			Let $\gamma>0$ and $\Om\subset\mb R^n$ be a bounded $C^{1,1}$ domain. Then the purely singular problem $(\mc P_{\gamma,0})$ has a unique weak solution $u_0 \in H^1_{\mathrm{loc}}(\Om)\cap L^1(\Om)\cap C(\Om)$ such that
			\begin{enumerate}
				\item [(i)] $u_0 \in H_0^1(\Om)\cap L^\infty(\Om)$ if $ 0 < \gamma \leq 1$, with $\inf_K u_0 >0$ for any $ K \Subset \Om$.

				\item[(ii)] $u_0 \in H^1_{\text{loc}}(\Om) \cap L^\infty(\Om) $ such that $u_0^\frac{\gamma+1}{2} \in H^1_0(\Om)$ if $\gamma >1$, with $\inf_K u_0 >0$ for any $ K \Subset \Om$.
			\end{enumerate}
			Moreover, we have 
			\begin{align}\label{e3.1}
		\|u_1\|_{L^\infty(\Om)}^{\frac{-\gamma}{\gamma+1}} u_1 \leq u_0 \leq ((\gamma +1)u_1)^\frac{1}{\gamma+1},
			\end{align}
			where $u_1\in L^\infty (\Om)\cap C^{1,\ba}(\bar{\Om})~\text{for every}~ \ba \in (0,1)$ is the unique solution to $\mc M u =1$ and $u =0$ in $\mb R^n \setminus \Om$. In particular, $u_0 \in C(\overline{\Om})$. 
		\end{prop}
		\begin{proof}
			The existence, uniqueness and summability properties of the weak solution $u_0$ are established in \cite[Theorems 2.13, 2.14, 2.15, and 2.16]{GUna}. Next, we proceed to prove equation \eqref{e3.1}. From \cite[Lemma 3.1]{AroRad}, we know that the problem $\mc Mu =1$ in $\Om$, and $u =0$ in $\mb R^n \setminus \Om$, has the unique solution 
			\begin{align}\label{eu1}
				u_1 \in L^\infty (\Om)\cap C^{1,\ba}(\bar{\Om})~\text{for every}~ \ba \in (0,1)~\text {such that}~ u_1 >0~\text {in}~\Om.
			\end{align} 
			Next, we define
			\begin{align}\label{e3.2}
				v =	\|u_1\|_{L^\infty(\Om)}^{\frac{-\gamma}{\gamma+1}}u_1~\text{and}~V = ((\gamma +1)u_1)^\frac{1}{\gamma+1}.
			\end{align}                                   
			We observe that $v \leq V$ in $\Om$. Now, we first show that $v$ is a weak subsolution to problem $(\mc P_{\gamma,0})$. For this let $\varphi \in C_c^\infty(\Om)$ with $\varphi \geq 0$. Then we have
			\begin{align*}
				\I\Om& \na v \na \varphi + \iint_{\mb R^{2n}}  \mathcal{B}(v)(x,y)\mathcal{B}(\varphi)(x,y)\,d\nu\\
				&=\|u_1\|_{L^\infty(\Om)}^{\frac{-\gamma}{\gamma+1}}	\left(\I\Om \na u_1 \na \varphi + \iint_{\mb R^{2n}} \mathcal{B}(u_1)(x,y)\mathcal{B}(\varphi)(x,y)\,d\nu \right)\\
                &=\|u_1\|_{L^\infty(\Om)}^{\frac{-\gamma}{\gamma+1}} \I\Om \varphi dx\leq \I\Om {\varphi}{v^{-\gamma}}dx.
			\end{align*}
			Next, we show that $V$ is a weak supersolution to $(\mc P_{\gamma,0})$. Again for $\varphi \in C_c^\infty(\Om)$ with $\varphi \geq 0$, we have
			\begin{align}\label{e3.3}
				\I\Om \na V \na \varphi dx &= \I \Om  \na u_1 V^{-\gamma} \na \varphi dx \geq  \I{\Om}\na u_1 (V^{-\gamma}\nabla\varphi + \varphi \na V^{-\gamma})dx=\I\Om \na u_1 \na (V^{-\gamma}\varphi)dx.
			\end{align}
			We also have:
			\begin{align}\nonumber \label{e3.4}
				\iint_{\mb R^{2n}} \mathcal{B}(V)(x,y)\mathcal{B}(\varphi)(x,y)\,d\nu =&\iint_{\mb R^{2n} \cap \{\varphi(x)\geq\varphi(y)\}}\mathcal{B}(V)(x,y)\mathcal{B}(\varphi)(x,y)\,d\nu\\
				&+ \iint_{\mb R^{2n} \cap \{\varphi(x)<\varphi(y)\}}\mathcal{B}(V)(x,y)\mathcal{B}(\varphi)(x,y)\,d\nu.
			\end{align}
			We now estimate first term on the R.H.S. of \eqref{e3.4}. Since the map $t \ra t^{\frac{1}{\gamma+1}}$ for $t >0$ and $\gamma >0$ is concave, we deduce that
			\begin{align}\label{e3.5}\nonumber
				\iint_{\mb R^{2n} \cap \{\varphi(x)\geq\varphi(y)\}}\mathcal{B}(V)(x,y)&\mathcal{B}(\varphi)(x,y)\,d\nu\\ \nonumber
				\geq&	\iint_{\mb R^{2n} \cap \{\varphi(x)\geq\varphi(y)\}}V^{-\gamma}(x)\mathcal{B}(u_1)(x,y)\mathcal{B}(\varphi)(x,y)\,d\nu\\ \nonumber
				\geq& \iint_{\mb R^{2n} \cap \{\varphi(x)\geq\varphi(y)\}}\mc B(u_1)(x,y)\mc B(V^{-\gamma}\varphi)(x,y)d\nu \\ \nonumber
				&- \iint_{\mb R^{2n} \cap \{\varphi(x)\geq\varphi(y)\}}\mc B(u_1)(x,y)(V^{-\gamma}(x)-V^{-\gamma}(y))\varphi(y)d\nu\\ 
				\geq& \iint_{\mb R^{2n} \cap \{\varphi(x)\geq\varphi(y)\}}\mc B(u_1)(x,y)\mc B(V^{-\gamma}\varphi)(x,y)d\nu. 
			\end{align}
			By symmetry, using the same argument, we obtain:
			\begin{align} \label{e3.6}
				\iint_{\mb R^{2n} \cap \{\varphi(x)<\varphi(y)\}}&\mathcal{B}(V)(x,y)\mathcal{B}(\varphi)(x,y)\,d\nu
				\geq & \iint_{\mb R^{2n} \cap \{\varphi(x)<\varphi(y)\}}\mc B(u_1)(x,y)\mc B(V^{-\gamma}\varphi)(x,y)d\nu.
			\end{align}
			Using \eqref{e3.5} and \eqref{e3.6} in \eqref{e3.4} and combining it with \eqref{e3.3}, we obtain that $V$ is a weak supersolution to the problem $(\mc P_{\gamma,0})$, as claimed. Now using the definitions in \eqref{e3.2} and Lemma \ref{tcp}, we obtain \eqref{e3.1}.\\
			Finally, based on equation \eqref{e3.1}, we conclude that $u_0 \in C (\overline{\Om})$ iff $u_1 \in C(\overline{\Om})$ and $u_1 =0$ in $\mb R^n \setminus \Om$, which is indeed the case, as shown in equation \eqref{eu1}.
		\end{proof}
\subsubsection{Symmetry of $u_0$}
We have the following result in this direction.
\begin{prop}\label{pspu_0}
	Let $u_0 \in C(\overline{\Om})$ be the unique weak solution of $(\mc P_{\gamma,0})$ given by Proposition \ref{p3.1}. Then, for any $-1<\la<\la^\ast$, we have 
	\begin{align}\label{4.2}
		u_0(x) \leq u_{0_\la}(x),~\forall x \in \Om_\la.
	\end{align}
\end{prop}
\begin{proof}
	By \cite[Lemma 3.2]{AroRad}, let $u_k \in H_0^1(\Om) \cap C^{2}(\overline{\Om})$ be a weak solution to
	\begin{align}\label{4.3}
		\begin{cases}
			\mc M u_k ={\left(u_k+\frac{1}{k}\right)^{-\gamma}}~&{in\,\,} \Om,\\
			u_k >0~&{in\,\,} \Om,\\
			u_k=0~&{in\,\,} \mb R^n \setminus\Om.
		\end{cases}
	\end{align}
	Therefore, we can apply the moving plane method in the same way as in \cite{BVDDS} to conclude that \eqref{4.2} holds for each $u_k$. From \cite{AroRad}, we know that $u_k\to u_0$ in $\Om$ as $k\to\infty$, and thus, the result also holds for $u_0$.
	\end{proof}
\subsection{Properties of $z$}
\subsubsection{Comparison Principles}
To apply the moving plane method and establish the symmetry properties of $z$, we first prove some comparison results.
	\begin{prop}\label{p4.2}
		Let $\gamma>0$ and $u \in H^1_{\mathrm{loc}}(\Om)\cap L^1(\Om)\cap C(\overline{\Om})$ be a weak solution to the problem $(\mc P_{\gamma,w})$.
        Let $z$ be defined by \eqref{deco}. Then it follows that
		\begin{align*}
			z>0~\text{in}~\Om.
		\end{align*}
	\end{prop}
\begin{proof}
	Since $u \in C(\overline{\Om})$ and by Proposition \ref{p3.1}, we have $u_0 \in C(\overline{\Om})$, therefore, we obtain $z \in H_0^1(\Om) \cap C(\overline{\Om})$. Also by the hypothesis on $w$, it follows that $u$ is a weak supersolution of the equation 
	\begin{align*}
		\mc M v =v^{-\gamma}\text { in }\Om.
	\end{align*}
	Therefore by Lemma \ref{tcp}, it follows that
	\begin{align*}
		u \geq u_0~\text{in}~\Om~\text{and therefore}~z \geq 0~\text{in}~\Om.
	\end{align*}
	Next we show that $z>0$ in $\Om$. Assume there exists $x_0 \in \Om$ such that $z(x_0) =0$. We claim that
	\begin{align}\label{eclaim}
\text{	there exists}~ r>0~ \text{such that}~ z \equiv 0~\text{ on}~ B_r(x_0).
	\end{align} 
	 For this choose $R>0$ such that $B_R(x_0) \Subset \Om$. Now we show that $z$ is a weak supersolution to 
	\begin{align}\label{esupsol}
		\mc Mv + \Lambda v =0~\text{in}~B_R(x_0),
	\end{align}
	for some $\La >0$ in the sense of \cite[Definition 3.1]{BV}. For this let $ \mc O \Subset B_R(x_0)$ and $\varphi \in \chi_+^{1,2}(\mc O)$ (see \cite[page 6]{BV} for its definition). This means $ \varphi \geq 0$ on $\mc O$, $\varphi =0$ on $\mb R^n \setminus \mc O$ and $\varphi_{\vert_{\mc O}} \in H^1_0(\mc O)$. Again, since $\mc O \Subset B_R(x_0) \Subset \Om$ and $\varphi =0$ in $\mb R^n \setminus \mc O$, we conclude that $ \varphi \in H_c^1(\Om)$. Now
	\begin{equation}\label{e4.3}
    \begin{split}
		&\I{\mc O}\na z \na \varphi dx + \iint_{\mb R^{2n}}\mathcal{B}(z)(x,y)\mathcal{B}(\varphi)(x,y)\,d\nu\\
&		 =	\I{\Om} \na z \na \varphi dx + \iint_{\mb R^{2n}}\mathcal{B}(z)(x,y)\mathcal{B}(\varphi)(x,y)\,d\nu\\
		  &= \I\Om\left({(u_0+z)^{-\gamma}}+\wp(u)-{u_0^{-\gamma}}\right)\varphi dx\\
          &=\I{\mc O}\left({(u_0+z)^{-\gamma}}+\wp(u)-{u_0^{-\gamma}}\right)\varphi dx\geq  \I{\mc O}\left({(u_0+z)^{-\gamma}}-{u_0^{-\gamma}}\right)\varphi dx.
		 	\end{split}
            \end{equation}
	Using Proposition \ref{p3.1}, there exists a constant $C_{B_R(x_0)}>0$ such that
	\begin{align}\label{e4.4}
	u_0 \geq C_{B_R(x_0)}>0~\text{on}~B_R(x_0).
	\end{align}
	Now using Mean value theorem and \eqref{e4.4}, we infer that
	\begin{align*}
	{(u_0+z)^{-\gamma}}-{u_0^{-\gamma}} =a(x)z,
	\end{align*}
	for some bounded coefficient $a(x)$ which depends only on $B_R(x_0)$. Thus we can find $\La >0$ independent of $\mc O$ such that
	\begin{align}\label{e4.5}
		{(u_0+z)^{-\gamma}}-{u_0^{-\gamma}} +\La z \geq 0.
		\end{align}
	Combining \eqref{e4.3} and \eqref{e4.5} we obtain
	\begin{align*}
		\mc  Mz +\La z \geq 0,~ \forall \varphi \in \chi_+^{1,2}(\mc O)~\text{and}~\forall \mc O \Subset B_R(x_0).
	\end{align*}
	This proves that $z$ is a weak supersolution of \eqref{esupsol}. Furthermore since $z \geq 0$ on $\Om$, we are entitled to apply \cite[Proposition 3.3]{BV} and hence we have the required claim \eqref{eclaim}.\\
	Finally by a covering argument we infer that $z\equiv 0$ on $\Om$ which implies $\wp (\cdot) =0$ and we get a contradiction.
	\end{proof}
	Next we give a weak comparison principle for the narrow domains.
	\begin{prop}\label{p4.3}
		Let $\gamma>0$, $\la \in (-1,\la^\ast)$ and $\tilde{\Om} \subset \Om_\la$. Assume that $u \in H^1_{\mathrm{loc}}(\Om)\cap L^1(\Om)\cap C(\overline{\Om})$ is a weak solution of $(\mc P_{\gamma,w})$. Let $z$ be given by \eqref{deco} and suppose that
		\begin{align*}
			z \leq z_\la~\text{on}~\partial \tilde{\Om}.
		\end{align*}
		Then there exists a positive constant $\de =\de(u,\wp)$ such that, if $|\tilde{\Om}| \leq \delta$, then
		\begin{align*}
			z \leq z_\la~\text{in}~\tilde{\Om}.
		\end{align*}
	\end{prop}
	\begin{proof}
		We have
		\begin{align}\label{e4.8}
			\mc M(u_0 +z) ={(u_0+z)^{-\gamma}}+\wp(u_0 +z)~\text{in}~\Om,
		\end{align}
		and
			\begin{align}\label{e4.9}
			\mc M(u_{0_\la} +z_\la) ={(u_{0_\la}+z_\la)^{-\gamma}}+\wp(u_{0_\la} +z_\la)~\text{in}~\Om.
		\end{align}
	From the given condition, since $z \leq z_\la~\text{on}~\partial \tilde{\Om}$, we have $(z-z_\la)^+ \in H_0^1(\tilde{\Om})$ and so we can consider a sequence of positive functions $\{\varphi_k\}_{k \in \mb N}$ such that
		\begin{align*}
			\varphi_k \in C_c^\infty(\tilde{\Om})~\text{and}~\varphi_k \ra (z-z_\la)^+~\text{in}~H_0^1(\tilde{\Om}).
		\end{align*}
		We can also assume that supp\,$\varphi_k \subset$ supp\,$(z-z_\la)^+$. Test \eqref{e4.8} and \eqref{e4.9} with $\varphi_k$ and subtracting, we get
		\begin{align}\nonumber\label{e4.10}
		\I{\tilde{\Om}}\na(u_0 +z)-\na(u_{0_\la}+z_\la)dx +\iint_{\mb R^{2n}} (\mc B(u_0+z)-\mc B(u_{0_\la}+z_\la)(x,y))\mc B(\varphi_k)(x,y)d\nu\\
		=\I{\tilde{\Om}}\left({(u_0+z)^{-\gamma}}+\wp(u_0 +z)-{(u_{0_\la}+z_\la)^{-\gamma}}-\wp(u_{0_\la} +z_\la)\right)\varphi_k dx.
			\end{align}
	Since $u_0$ and $u_{0_\la}$ solve $(\mc P_{\gamma,0})$, we deduce from \eqref{e4.10} that
	\begin{align}\label{e4.11}\nonumber
	&\I{\tilde{\Om}} \na(z-z_\la)\na \varphi_k dx +\iint_{\mb R^{2n}} \mc B(z-z_\la)(x,y)\mc B (\varphi)(x,y)d \nu\\		 &=\I{\tilde{\Om}}\left({u_{0_\la}^{-\gamma}}-{u_{0}^{-\gamma}}+{(u_0+z)^{-\gamma}}-{(u_{0_\la}+z_\la)^{-\gamma}}\right) \varphi_k dx +\I{\tilde{\Om}}\left(\wp(u_0 +z)-\wp(u_{0_\la} +z_\la)\right)\varphi_k dx.
	\end{align}
	Since $u_0 \leq u_{0_\la}$ in $\Om_\la$ (see \eqref{4.2}) and $z \geq z_\la$ on the support of $\varphi_k$, by applying Lemma \ref{l4.1} with $x =u_0$, $y=z$, $j =u_{0_\la}$ and $m =z_\la$ we get
	\begin{align*}
	u_0^\gamma(u_0+z)^\gamma(u_{0_\la}+z_\la)^\gamma + u_0^\gamma u_{0_\la}^\gamma(u_{0_\la}+z_\la)^\gamma-	u_{0_\la}^\gamma(u_0+w)^\gamma(u_{0_\la}+z_\la)^\gamma -u_0^\gamma u_{0_\la}^\gamma(u_{0}+z)^\gamma \leq  0,
	\end{align*}
	and so 
	\begin{align}\label{4.14}
	{u_{0_\la}^{-\gamma}}-{u_{0}^{-\gamma}}+{(u_0+z)^{-\gamma}}-{(u_{0_\la}+z_\la)^{-\gamma}} \leq 0.
	\end{align}
	Therefore, by using the assumption $(\mc A)$ and \eqref{4.14} in \eqref{e4.11}, we can find a constant $C>0$ such that
	\begin{align*}
	&\I{\tilde{\Om}} \na(z-z_\la)\na \varphi_k dx +\iint_{\mb R^{2n}} \mc B(z-z_\la)(x,y)\mc B (\varphi)(x,y)d \nu\\
	 &\leq 	\I{\tilde{\Om}}\left(\wp(u_0 +z)-\wp(u_{0_\la} +z_\la)\right)\varphi_k dx\\
&	 \leq \I{\tilde{\Om}}\left(\wp(u_{0_\la} +z)-\wp(u_{0_\la} +z_\la)\right)\varphi_k dx\leq C \I{\tilde{\Om}} |z-z_\la|\varphi_k\,dx.
	\end{align*}
	We now pass to the limit $k \ra \infty$ in above inequality to get
	\begin{align*}
	\I{\tilde{\Om}} | \na(z-z_\la)^+|^2dx \leq 	\I{\tilde{\Om}} | \na(z-z_\la)^+|^2dx +[(z-z_\la)^+]^2_{s,\mb R^n} \leq C\I{\tilde{\Om}}|(z-z_\la)^+|^2dx
	\end{align*}
	and finally by the Poincar\'{e} inequality, we get
	\begin{align*}
	\I{\tilde{\Om}} | \na(z-z_\la)^+|^2dx \leq C C'(\tilde{\Om}) \I{\tilde{\Om}} | \na(z-z_\la)^+|^2dx,
	\end{align*}
where $C^\prime(\tilde{\Om}) \ra 0$ as $|\tilde{\Om}| \ra 0$. Thus there exists  $\delta$ small such that $|\tilde{\Om}|<\delta$ implies $CC'(\tilde{\Om}) <1$ and so $(z-z_\la)^+ =0$ in $\tilde{\Om}$. This completes the proof.
\end{proof}
	In the following lemma, we give a proof of a Strong Comparison Principle.
	\begin{Lemma}\label{l4.4}
		Let $u \in H^1_{\mathrm{loc}}(\Om)\cap L^1(\Om)\cap C(\overline{\Om})$ be a weak solution to problem $(\mc P_{\gamma,w})$ with $\wp(\cdot)$ satisfying $(\mc A)$. Let $z$ be given by \eqref{deco} and assume that for some $\la \in (-1,\la^\ast)$, we have
		\begin{align*}
			z \leq z_\la~\text{in}~\Om_\la.
			\end{align*}
			Then $z<z_\la$ in $\Om_\la$ unless $z \equiv z_\la$ in $\Om_\la$.
	\end{Lemma}
	\begin{proof}
		Let us assume that there exist $x_0 \in \Om_\la$ such that $z(x_0) =z_\la(x_0)$ and let $R =R(x_0)>0$ such that $B_R(x_0) \Subset \Om_\la$ and $B_R(x_0) \Subset \Om$. Letting $\omega_\la =z_\la -z$, we claim that
		\begin{align}\label{Ec}
			 \omega_\la ~\text{is a weak supersolution of \eqref{esupsol} in the sense of \cite[Definition 3.1]{BV}}.
			 \end{align}
			  For this let $\mc O \Subset B_R(x_0)$ and $\varphi \in \chi_+^{1,2}(\mc O)$ (where $\chi_+^{1,2}(\mc O)$ is as in Proposition \ref{p4.2}). Again as in Proposition \ref{p4.2} we have $\varphi \in H_c^1(\Om)$ with $\varphi \geq 0$ in $\Om$. Now
		\begin{align}\nonumber\label{e4.14}
			\I{\mc O}& \na \omega_\la \na \varphi dx + \iint_{\mb R^{2n}}\mathcal{B}(\omega_\la)(x,y)\mathcal{B}(\varphi)(x,y)\,d\nu
			=\I{\Om} \na \omega_\la \na \varphi dx + \iint_{\mb R^{2n}}\mathcal{B}(\omega_\la)(x,y)\mathcal{B}(\varphi)(x,y)\,d\nu\\ \nonumber
			=& \I{\mc O} \left({u_0^{-\gamma}}-{u_{0_\la}^{-\gamma}}+{(u_{0_\la}+z)^{-\gamma}}-{(u_0+z)^{-\gamma}}\right)\varphi dx + \I{\mc O} (\wp(u_{0_\la}+z_\la)-\wp(u_0+z)\varphi dx\\
			&+\I{\mc O} \left({(u_{0_\la}+z_\la)^{-\gamma}}-{(u_{0_\la}+z)^{-\gamma}}\right)\varphi dx.
		\end{align}
		Since for $0< a\leq b$ the function $h(s):=a^{-\gamma}-b^{-\gamma}+(b+s)^{-\gamma}-(a+s)^{-\gamma}$ is increasing in $[0,\infty)$, we have
		\begin{align}\label{e4.15}
			\left({u_0^{-\gamma}}-{u_{0_\la}^{-\gamma}}+{(u_{0_\la}+z)^{-\gamma}}-{(u_0+z)^{-\gamma}}\right) \geq 0.
		\end{align}
		Moreover since $\wp$ is nondecreasing, $u_0 \leq u_{0_\la}$ (see \eqref{4.2}) in $\Om_\la$ and $z\leq z_\la$, we have
		\begin{align}\label{e4.16}
			\wp(u_{0_\la}+z_\la)-\wp(u_0+z) \geq 0.
			\end{align}
			Using \eqref{e4.15} and \eqref{e4.16} in \eqref{e4.14} we obtain
			      \begin{equation}\label{e4.17}
                \begin{split}
					\I{\mc O} \na \omega_\la \na \varphi dx + \iint_{\mb R^{2n}}\mathcal{B}(\omega_\la)(x,y)\mathcal{B}(\varphi)(x,y)\,d\nu&\geq \I{\mc O} \left({(u_{0_\la}+z_\la)^{-\gamma}}-{(u_{0_\la}+z)^{-\gamma}}\right)\varphi dx.
                \end{split}
				\end{equation}
				Since $u_{0_\la} \geq u_0 \geq C_{B_R(x_0)}>0$ (see Proposition \ref{p3.1}), by arguing as in Proposition \ref{p4.2}, we find $\La >0$ independent of $\mc O$ such that
				\begin{align*}
				{(u_{0_\la}+z_\la)^{-\gamma}}-{(u_{0_\la}+z)^{-\gamma}} +\La \omega_\la \geq 0.
				\end{align*}
				This in combination with \eqref{e4.17} proves our claim \eqref{Ec}. Thus by \cite[Proposition 3.3]{BV} there exists $r>0$ such that $\omega_\la \equiv 0$ in $B_r(x_0)$, and by a covering argument $z_\la \equiv z$ in $\Om_\la$. This completes the proof.
		\end{proof}
		\subsubsection{Symmetry of $z$}
		\begin{prop}\label{p4.5}
			Let $\gamma>0$ and $u \in H^1_{\mathrm{loc}}(\Om)\cap L^1(\Om)\cap C(\overline{\Om})$ be a weak solution to $(\mc P_{\gamma,w})$. Assume that $z$ is given by \eqref{deco}. Then for any $\la \in (-1,\la^\ast)$ we have
			\begin{align*}
				z(x) < z_\la(x),~ \forall x \in \Om_\la.
			\end{align*}
		\end{prop}
		\begin{proof}
			Let $\la >-1$. Since ${z} >0$ in $\Om$ by Proposition \ref{p4.2}, we have
			\begin{align*}
				z \leq z_\la~\text{on}~\partial \Om_\la.
			\end{align*}
			Therefore assuming $\la$ close to $-1$ we have $|\Om_\la|$ is sufficiently small, so we are entitled to apply Proposition \ref{p4.3} to get 
			\begin{align}\label{e4.22}
				z \leq z_\la~\text{in}~\Om_\la,~\text{with}~\la~\text{sufficiently close to}~-1,
				\end{align}
				and finally by the Strong Comparison Principle (Lemma \ref{l4.4}) we have $z<z_\la$ in $\Om_\la$, with $\la$ sufficiently close to $-1$.\\
				Let us define 
				\begin{align*}
					\La_0 = \{\la>-1:z \leq z_{\tilde{\la}}~\text{in}~\Om_{\tilde{\la}}~\text{for all}~\tilde{\la}\in(-1,\la)\},
				\end{align*}
				which is not empty thanks to \eqref{e4.22}. Let us set
				\begin{align*}
					\la_0 = \sup \La_0.
				\end{align*}
				Note that to prove our result we have to show that actually $\la_0 =\la^\ast$. On the contrary suppose that $\la_0 <\la^\ast$. Then by continuity, we obtain $z \leq z_{\la_0}$ in $\Om_{\la_0}$. Then in view of Lemma \ref{l4.4}, we have either $z < z_{\la_0}$ in $\Om_{\la_0}$ or $z=z_{\la_0}$ in $\Om_{\la_0}$. But $ z=z_{\la_0}$ is not possible because of the zero Dirichlet boundary condition and the fact that $z >0$ in $\Om$ from Proposition \ref{p4.2}. Thus $z<z_{\la_0} $ holds in $\Om_{\la_0}$.\\ 
				Now consider $\de$ given by Proposition \ref{p4.3}, so that the weak comparison principle holds true in any subdomain $\tilde{\Om}$ if $|\tilde{\Om} |<\delta$. Fix a compact set $ K \subset \Om_{\la_0}$ so that $|\Om_{\la_0}\setminus K| \leq \frac{\de}{2}$. By compactness we can find $\mu >0$ such that 
				\begin{align*}
					z_{\la_0} -z \geq 2\mu >0~\text{in}~K.
				\end{align*}
				Take now $\tilde{\e}>0$ sufficiently small so that $\la_0+\tilde{\e} <\la^\ast$ and for any $0<\e \leq \tilde{\e}$ we have
				\begin{enumerate}
					\item [$(i)$] $z_{\la_0+\e}-z \geq 0~\text{in}~K,$
					\item [$(ii)$] $|\Om_{\la_0+\e} \setminus K| \leq \de$.
				\end{enumerate}
				In view of $(i)$ above we infer that, for any $0<\e \leq \tilde{\e}$, $z \leq z_{\la_0+{{\e}}}$ on the boundary of $\Om_{\la_0+\e}\setminus K$. Consequently by $(ii)$, we can apply Lemma \ref{p4.3} and deduce that
				\begin{align*}
					z \leq z_{\la_0+\e}~\text{in}~\Om_{\la_0+\e} \setminus K.
				\end{align*}
				Thus $z \leq z_{\la_0+\e}$ in $\Om_{\la_0+\e}$ and applying Lemma \ref{l4.4} we have $z < z_{\la_0+\e}$ in $\Om_{\la_0+\e}$. This is a contradiction to the definition of $\la_0$ and we conclude that $\la_0=\la^\ast$. This completes the proof.
			\end{proof}

   \subsection{Proof of the symmetry result}
	\textbf{Proof of Theorem \ref{t1.3}:}
    {We observe that by assumption, $\la^\ast =0$. Therefore, by applying Proposition \ref{pspu_0} and \ref{p4.5} in the $x_1$-direction, we get
	\begin{align*}
		u_0(x)+z(x) \leq (u_0)_{\la^\ast}(x)+z_{\la^\ast}(x),\quad \forall x\in \Om_0,
	\end{align*}
	and in the $-x_1$-direction to get 
		\begin{align*}
		u_0(x)+z(x) \geq (u_0)_{\la^\ast}(x)+z_{\la^\ast}(x),\quad \forall x\in \Om_0.
	\end{align*}
	Thus $u(x)=u_{\la^\ast}(x)$ in $\Om$ and the proof is complete.} \qed
\section{Jumping problem}
First, we establish the existence result for a variational inequality in the following subsection.
\subsection{Existence for a class of singular variational inequalities}
Throughout this subsection, unless otherwise mentioned, we assume that $w\in H^{-1}(\Om)$ and $h:\Om\times\mb R\to\mb R$ is a Carath\'eodory function that satisfies conditions $(h_1^\prime)$ and $(h_2)$, where $(h_1^\prime)$ is stated as follows: 
\begin{enumerate}
	\item [$(h_1^\prime)$] there exist two functions $\theta:\Om\to\mb R$ and $\kappa:\Om\to\mb R$ such that
	\begin{align*}
		|h(x,t)| \leq |\theta(x)|+|\kappa(x)||t|,~\text{for}~x \in \Om~\text{and every}~t \in \mb R,
	\end{align*}
	where $\theta \in L^\frac{2n}{n+2}(\Om)$ and $\kappa \in L^\frac{n}{2}(\Om)$.
\end{enumerate}
Also, we recall the first eigenfunction $e_1$ of $\mc M$ defined by the equation \eqref{mevp} and its associated eigenfunction $\la_1$.
                In this subsection, we establish the existence result for the variational inequality given as follows:
					\begin{align}\label{e5.14}
					\begin{cases}	 u >0~\text{in}~\Om~\text{and}~u^{-\gamma} \in L^1_{\text{loc}}(\Om),\\
						\I\Om \na u \na (v-u)dx+\iint_{\mb R^{2n}}\mc B(u)(x,y)\mc B(v-u)(x,y)d\nu\\
						~~~~~~~\geq \I\Om \left(u^{-\gamma}+h(x,u)\right)(v-u)dx -\la\I\Om e_1(v-u)dx +\ld w,(v-u)\rd,\\
						~~~~~~~~~~~~~~~~~~~~~~~~~~~~~~~~~ \forall v \in u+(H_0^1(\Om)\cap L_c^\infty(\Om))~\text{with}~v \geq 0~\text{in}~ \Om, \\
					u \leq 0~\text{on}~\partial \Om.
					\end{cases}
				\end{align}
                More precisely, we establish the following result.
                \begin{Theorem}\label{t5.4}
					Assume that $\al > \la_1$. Then there exists $\overline{\la} \in \mb R$ such that for every $\la > \overline{\la}$, problem \eqref{e5.14} admits at least two distinct weak solutions in $H^1_{\text{loc}}(\Om)\cap L^1(\Om)$. 
				\end{Theorem}
                \subsubsection{Preliminaries}
Recalling $u_0$ as in Proposition \ref{p3.1}, we define $J_0:\Om \times \mb R \ra [0,+\infty]$ by
		\begin{align}\label{edj_0}
			J_0(x,t) = P(u_0(x)+t)-P(u_0(x)) +tu_0(x)^{-\gamma},
		\end{align}
        where 
		\begin{align}\label{edP}
			P(t) =\begin{cases}
				-\I1^t s^{-\gamma}ds~&\text{if}~s \geq0,  \\
				+\infty~&\text{if}~ s<0.
			\end{cases}
		\end{align}
		Note that $J_0(x,0) =0$ and $J_0(x,\cdot)$ is convex and lower semicontinuous for any $x \in \Om$. Also $J_0(x,\cdot)$ is $C^1$ on $(-u_0(x),+\infty)$ with
		\begin{align}\label{e3.10}
			D_t J_0(x,t)= u_0(x)^{-\gamma}-(u_0(x)+t)^{-\gamma}.
		\end{align}
Moreover, let $K : H_0^1(\Om) \ra \mb R \cup \{+\infty\}$ be the convex functional defined by
	\begin{align}\label{eDK}
		K(u)=\frac{1}{2 }\I\Om|\na u|^2dx + \frac{1}{2}\iint_{\mb R^{2n}}\frac{|u(x)-u(y)|^2}{|x-y|^{n+2s}}dxdy +\I\Om J_0(x,u)dx.
	\end{align}
    Now for any $\la \in \mb R$, let $\Psi_\la:H_0^1(\Om) \ra (-\infty,+\infty] $ be the functional defined as 
\begin{equation}\label{psi}
\Psi_\la =K +
\hat{H}_\la,
\end{equation}
with
\begin{align*}
\hat{H}_\la(u) =-\I\Om H_1(x,u)dx +\la \I\Om e_1 udx -\ld w,u\rd,
\end{align*}
where
\begin{equation}\label{hh}
H_1(x,t)=\int_{0}^{t}h_1(x,s)\,ds,\quad h_1(x,t)=h(x,u_0(x)+t),\quad \text{ for }x\in\Om,\,t\in\mb R. 
\end{equation}
Furthermore, we define the functional $\Phi_w:L^2(\Om)\to (-\infty,+\infty]$ by
		\begin{equation}\label{e3.11}
        \begin{split}
			\Phi_w(u) =\begin{cases}
            \frac12\I\Om |\na (u-u_0)|^2dx + \frac12 \iint_{\mb R^{2n}} |\mc  B(u-u_0)(x,y)|^2d\nu\\
			+\I\Om J_0(x,u-u_0)dx -\ld w,u-u_0 \rd,~&\text{if}~u \in u_0 + H_0^1(\Om), \\
             +\infty,~& \text{ otherwise}.
            \end{cases}
            \end{split}
            \end{equation}
        We note that $\Phi_w$ is strictly convex, lower semiconitnuous and coercive and that $\Phi_w(u_0)=0$. Also note that the domain of the functional $\Phi_w$ is given by
		\begin{align*}
			\{u \in u_0 + H_0^1(\Om):J_0(x,u-u_0)\in L^1(\Om)\}.
		\end{align*}
We next prove a lemma which will be crucial to show that $\Psi_\la$ satisfies the (PS) condition. 
\begin{Lemma}\label{l5.1} 
	Let $\{u_k\}_{k \in \mb N}$ be a sequence in $H_0^1(\Om)$ and $\{\omega_k\}_{k \in \mb N}$ be a sequence in $H^{-1}(\Om)$. Suppose that $\{\omega_k\}_{k \in \mb N}$ is strongly convergent in $H^{-1}(\Om)$ and that
	\begin{align}\nonumber\label{e5.1}
		\frac{1}{2}&\I\Om |\na v|^2dx+\frac12 \iint_{\mb R^{2n}}|\mc B(v)(x,y)|^2d\nu +\I\Om J_0(x,v)dx\\
		& \geq \frac{1}{2}\I\Om |\na u_k|^2dx+\frac12 \iint_{\mb R^{2n}}|\mc B(u_k)(x,y)|^2d\nu +\I\Om J_0(x,u_k)dx+\ld \omega_k,v-u_k \rd,~\forall v \in H_0^1(\Om).
	\end{align}
	Then $\{u_k\}_{k \in \mb N}$ is strongly convergent in $H_0^1(\Om)$.
\end{Lemma}
\begin{proof}
	Taking $v =0$ in \eqref{e5.1}, we get
	\begin{align*}
		\frac{1}{2}\I\Om |\na u_k|^2dx+\frac12 \iint_{\mb R^{2n}}|\mc B(u_k)(x,y)|^2d\nu +\I\Om J_0(x,u_k)dx \leq \ld \omega_k,u_k\rd,
	\end{align*}
	which implies that $\{u_k\}_{k \in \mb N}$ is a bounded sequence in $H_0^1(\Om)$. This further implies that up to a subsequence $u_k \rightharpoonup u$ weakly in $H_0^1(\Om)$ with $J_0(x,u)\in L^1(\Om)$.\\
	If we put $v =u$ in \eqref{e5.1}, we obtain
	\begin{align*}
		\limsup_{k \ra \infty} &\left(\frac{1}{2}\I\Om |\na u_k|^2dx+\frac12 \iint_{\mb R^{2n}}|\mc B(u_k)(x,y)|^2d\nu +\I\Om J_0(x,u_k)dx\right)\\
		& \leq \frac{1}{2}\I\Om |\na u|^2dx+\frac12 \iint_{\mb R^{2n}}|\mc B(u)(x,y)|^2d\nu +\I\Om J_0(x,u)dx.
	\end{align*}
	Since $J_0(x,t) \geq 0$ using Fatou's lemma, we infer that
	\begin{align*}
		\limsup_{k \ra \infty} &\left(\frac{1}{2}\I\Om |\na u_k|^2dx+\frac12 \iint_{\mb R^{2n}}|\mc B(u_k)(x,y)|^2d\nu\right) \leq \frac{1}{2}\I\Om |\na u|^2dx+\frac12 \iint_{\mb R^{2n}}|\mc B(u)(x,y)|^2d\nu.
		\end{align*}
		This in combination with Remark \ref{rm2.4} and the lower semicontinuity of the norm implies that $u_k \ra u$ strongly in $H_0^1(\Om)$ up to a subsequence. Actually all the sequence $\{u_k\}_{k\in \mb N}$ converges to $u$ in $H_0^1(\Om)$. Indeed, assuming $\omega_k \ra \omega$ and passing to the limit in \eqref{e5.1}, we get
		\begin{align*}
			\frac{1}{2}&\I\Om |\na v|^2dx+\frac12 \iint_{\mb R^{2n}}|\mc B(v)(x,y)|^2d\nu +\I\Om J_0(x,v)dx\\
			& \geq \frac{1}{2}\I\Om |\na u|^2dx+\frac12 \iint_{\mb R^{2n}}|\mc B(u)(x,y)|^2d\nu +\I\Om J_0(x,u)dx+\ld \omega,v-u \rd,~\forall v \in H_0^1(\Om),
		\end{align*}
		which means that $u$ is the minimum of a strictly convex functional $K -\omega$. Since the minimum of $K-\omega$ is unique, we conclude that the whole sequence $\{u_k\}_{k \in \mb N}$ converges to $u$ in $H_0^1(\Om)$.
	\end{proof}
	\begin{Lemma}\label{t5.2}
		Assume $\al >\la_1$. Then, for every $\la \in \mb R$, the functional $\Psi_\la$ satisfies the (PS) condition.
	\end{Lemma}
	\begin{proof}
		Let $\{u_k\}_{k\in \mb N}$ be a sequence in $H_0^1(\Om)$ and $\{\omega_k\}_{k\in \mb N}$ a sequence in $H^{-1}(\Om)$ with $$\sup_k|\Psi_\la(u_k)|<+\infty,\quad \omega_k \ra 0$$ and
        {$$
        K(v)\geq K(u_k)-\ld \hat{H}_\la^{'}(u_k),v-u_k\rd+\ld \omega_k,v-u_k\rd,\quad \forall v\in H_0^{1}(\Om),
        $$}
        that is
		\begin{align*}
				\frac{1}{2}&\I\Om |\na v|^2dx+\frac12 \iint_{\mb R^{2n}}|\mc B(v)(x,y)|^2d\nu +\I\Om J_0(x,v)dx\\ 
			 \geq& \frac{1}{2}\I\Om |\na u_k|^2dx+\frac12 \iint_{\mb R^{2n}}|\mc B(u_k)(x,y)|^2d\nu +\I\Om J_0(x,u_k)dx+ \I\Om(h_1(x,u_k)-\la e_1)(v-u_k)dx\\
			& +\ld w+\omega_k,v-u_k\rd,~\forall v \in H^1_0(\Om).
		\end{align*}
		This implies that for each $k$, {$u_k+u_0$ is the minimum of the energy functional $\Phi_{w+\omega_k}$ defined in \eqref{e3.11}}. Since $h(x,{u_k}+u_0)-\la e_1 +w+\omega_k \in H^{-1}(\Om)$ using Theorem \ref{t3.2}, we conclude that
		\begin{align}\label{e5.3}
			\begin{cases}	u_0 + u_k >0~\text{in}~\Om~\text{and}~(u_0 + u_k)^{-\gamma} \in L^1_{\text{loc}}(\Om),\\
				\I\Om \na u_k \na vdx+\iint_{\mb R^{2n}}\mc B(u_k)(x,y)\mc B(v)(x,y)d\nu\\
				~~~~~~~\geq \I\Om \left((u_0+u_k)^{-\gamma}-u_0^{-\gamma}\right)vdx +\I\Om(h_1(x,u_k)-\lambda\,e_1)vdx +\ld w+\omega_k,v\rd\\
				~~~~~~~~~~~~~~~~~~~~~~~~~~~~~~~~~ \forall v \in (H_0^1(\Om) {\cap} L_c^\infty(\Om))~\text{with}~v \geq -u_0-u_k~\text{in}~ \Om, \\
				u_0+u_k \leq 0~\text{on}~\partial \Om.
			\end{cases}
		\end{align}
		We first claim that $\{u_k\}_{k \in \mb N}$ is bounded in $H_0^1(\Om)$. Suppose on the contrary that
		\begin{align*}
		r_k:=\|u_k\| \ra +\infty
		\end{align*} 
		and let $z_k =u_k/r_k$. Then, up to a subsequence, $z_k \rightharpoonup \hat z$ weakly in $H_0^1(\Om)$ with $\hat z \geq 0$ in $\Om$.\\
		Using standard approximation argument, we can choose $v =-u_k$ in \eqref{e5.3} and get
		\begin{align*}
			\I\Om|\na u_k|^2dx +&\iint_{\mb R^{2n}}|\mc B(u_k)(x,y)|^2d\nu \\
			\leq &	\I\Om|\na u_k|^2dx + \iint_{\mb R^{2n}}|\mc B(u_k)(x,y)|^2d\nu+\I{\{u_k \geq 0\} } \left((u_0+u_k)^{-\gamma}-u_0^{-\gamma}\right)u_kdx\\ 
			\leq& \I\Om(h_1(x,u_k)-\lambda\,e_1)u_kdx +\ld w+\omega_k,u_k\rd,
		\end{align*}
		which implies
		\begin{align}\label{e5.4}\nonumber
			1=&
            \I\Om|\na z_k|^2dx +   \iint_{\mb R^{2n}}|\mc B(z_k)(x,y)|^2d\nu \\
			\leq& \I\Om \frac{h_1(z,r_kz_k)}{r_k}z_k-\frac{\la}{r_k}\I\Om e_1 z_k dx +\frac{1}{r_k}\ld w+\omega_k,z_k\rd.
		\end{align}
		We claim that $\hat z\not \equiv 0$ in $\Om$. To this end, letting $k \ra \infty$ in \eqref{e5.4}, we prove that
		\begin{align}\label{e5.5}
			1 \leq  \al \I \Om \hat z^2 dx.
		\end{align}
  Indeed letting $k \ra \infty$, we have
		\begin{align}\label{e5.6}
			\left|\frac{\la}{r_k}\I\Om e_1 z_k dx \right| \leq \frac{|\la|}{r_k}\|e_1\|_{L^2(\Om)}\|z_k\|_{L^2(\Om)}\leq C\frac{|\la|}{r_k}\|e_1\|_{L^2(\Om)}\|z_k\| \ra 0
		\end{align}
		and 
		\begin{align}\label{e5.7}
			\left|\frac{1}{r_k}\ld w +\omega_k, z_k\rd\right| \leq \frac{1}{r_k}\|w+z_k\|_{H^{-1}(\Om)} \ra 0.
		\end{align}
		Also by \cite[Lemma 3.3]{C}, we have that
		\begin{align}\label{e5.8}
		\lim_{k \ra \infty}	\frac{h_1(x,r_kz_k)}{r_k}=\alpha \hat z~\text{strongly in}~H^{-1}(\Om)
		\end{align}
		and so as $k \ra \infty$
		\begin{align}\label{e5.9}
			\I\Om \frac{h_1(x,r_kz_k)}{r_k} z_kdx \ra \I\Om \al \hat z^2dx.
		\end{align}
		Combining \eqref{e5.6}, \eqref{e5.7} and \eqref{e5.9} in \eqref{e5.4} we conclude that the claim \eqref{e5.5} holds. Now if we choose $v \in C_c^\infty(\Om)$ with $v \geq 0$ in $\Om$ in \eqref{e5.3} and divide by $r_k$, then using the definition of $z_k$, we get
		\begin{align*}
			&\I\Om \na z_k \na vdx+\iint_{\mb R^{2n}}\mc B(z_k)(x,y)\mc B(v)(x,y)d\nu\\
			&\geq \frac{1}{r_k}\I{\{u_k \geq 0\}} \left((u_0+u_k)^{-\gamma}-u_0^{-\gamma}\right)vdx +\frac{1}{r_k}\I\Om(h_1(x,r_kz_k)-\lambda\,e_1)vdx\\
            &\qquad+\frac{1}{r_k}\ld w+\omega_k,v\rd.
		\end{align*}
		Since $u_0 \geq C>0$ on the support of $v$, we can pass to the limit as $k \ra \infty$ and again using \eqref{e5.9}, we obtain
		\begin{align*}
			\I\Om \na \hat z \na vdx+&\iint_{\mb R^{2n}}\mc B(\hat z)(x,y)\mc B(v)(x,y)d\nu \geq \al\I\Om \hat z v dx,~\text{for every}~v \in C_c^\infty(\Om)~\text{with}~v \geq 0.
		\end{align*}
		 By density arguments, we put $v =e_1$ in above equation and using the fact that $\hat z \not \equiv 0$, we get $\la_1 \geq \al$, which is a contradiction to the our assumption that $\al >\la_1$. Thus $\{u_k\}_{k \in \mb N}$ is bounded in $H_0^1(\Om)$ and so up to a subsequence $u_k \rightharpoonup u$ weakly in $H_0^1(\Om)$. Then using the arguments in \cite[Theorems 4.36 and 4.37]{Groli}, by $(h_1^\prime)$, up to a subsequence $\{h_1(x,u_k)\}_{k\in \mb N}$ is strongly convergent to $h_1(x,u)$ in $H^{-1}(\Om)$. Now the assertion follows using Lemma \ref{l5.1}.
		\end{proof}
		The following theorem demonstrates that the functional $\Psi_\la$ indeed possesses the Mountain Pass geometry, as outlined in Theorem \ref{TMPT}.
		\begin{Lemma}\label{t5.3}
			Assume that $\al >\la_1$. Then the following facts hold:
			\begin{enumerate}
				\item [$(a)$] there exists $r, \overline{\la}, \mu >0$ such that $\Psi_\la(u) \geq \mu \la^2$ for every $\la > \overline{\la}$ and every $u \in H_0^1(\Om)$ with $\|u\|=\la r$.
				\item [$(b)$] there exists $v \in H_0^1(\Om)\cap L^\infty_c(\Om)$ such that $v \geq 0$ in $\Om$ and 
				\begin{align*}
					\lim_{t \ra +\infty} \Psi_\la(tv) =-\infty,\quad \forall~\la\in\mb R.
				\end{align*}
			\end{enumerate}
			\end{Lemma}
			\begin{proof}
				$(a)$ For every $\la >0$, let $\tilde{\Psi}_\la =\Psi_\la(\la u)/\la^2$ and define $\tilde{\Psi}_\infty:H_0^1(\Om) \ra (-\infty,+\infty]$ by
				\begin{align*}
					\tilde{\Psi}_\infty(u)=\begin{cases}
						\frac12 \I\Om |\na u|^2 dx+\frac{1}{2}[u]_{s,\mb R^n}^2-\frac{\al}{2}\I\Om u^2 dx+\I\Om e_1 udx~&\text{if}~u \geq 0~\text{in}~\Om,\\
						+\infty~&\text{otherwise}.
					\end{cases}
				\end{align*}
		We claim that there exists $r>0$ such that
		 \begin{align}\label{e5.10}
		 	\tilde{\Psi}_\infty(u) >0~\text{for every}~u \in H_0^1(\Om)~\text{with}~0<\|u\| \leq r.
		 	\end{align}	
		 	For this, let $ \mc H_+ =\{u \in H_0^1(\Om): u\geq 0~\text{in}~\Om\}$ and set
		 	\begin{align*}
		 	 \mc H_\infty=\left\{u \in \mc H_+:\frac12 \I\Om |\na u|^2 dx+\frac{1}{2}[u]_{s,\mb R^n}^ 2-\frac{\al}{2}\I\Om u^2 dx \leq 	\frac14 \I\Om |\na u|^2 dx+\frac{1}{4}[u]_{s,\mb R^n}^ 2 \right\}.
		 	\end{align*}
		 	In $\mc H_+\setminus \mc H_\infty$ the above claim trivially holds. On the other hand we claim that  
		 	\begin{align}\label{e5.11}
		 		c^{\prime}:=\inf\left\{\I\Om e_1 vdx: v \in \mc H_\infty, \|v\| =1\right\}>0.
		 	\end{align}
		 	Indeed otherwise, there exists a sequence $\{v_k\}_{k \in \mb N}$ in $\mc H_\infty$, $\|v_k\|=1$ and $\I\Om e_1v_kdx \ra 0$ as $k \ra +\infty$. Then up to a subsequence $v_k \rightharpoonup v$ weakly in $H_0^1(\Om)$, $v_k \ra v$ strongly in $L^2(\Om)$ and $v_k \ra v$ pointwise in $\Om$, implying that $v \geq 0$ in $\Om$. Now since $v_k \in \mc H_\infty$, we have
		 	\begin{align*}
		 		\I{\Om}v_k^2 dx \geq \frac{1}{2\al}.
		 	\end{align*}
		 	Using this fact and the strong convergence of $v_k$ in $L^2(\Om)$ we infer that $v \ne 0$ in $\Om$. On the other hand using the weak convergence we have $\I\Om v e_1 dx =0$. Since $ve_1 \geq 0$ in $\Om$ and $e_1>0$ in $\Om$, we have $v=0$ in $\Om$, a contradiction. Hence our claim \eqref{e5.11} is true. From here it is trivial to show that \eqref{e5.10} holds.\\
		 	Now by contradiction, suppose there exist a sequence $\{u_k\}_{k \in \mb N}$ in $H_0^1(\Om)$ and a sequence $ \la_k \ra +\infty$ with $\|u_k\| =r$ and
		 	\begin{align*}
		 		0 \geq& \limsup_{k} \tilde{\Psi}_{\la_k}(u_k)\\
		 		   =& \limsup_{k}  \left(\frac{1}{2 }\I\Om|\na u_k|^2dx + \frac{1}{2}\iint_{\mb R^{2n}} |\mc B(u_k)(x,y)|^2d\nu +\frac{1}{\la_k^2}\I\Om J_0(x,u)dx\right.\\
		 		   	&\left.-\I\Om \frac{H_1(x,\la_ku_k)}{\la_k^2}dx+\I\Om e_1u_kdx -\frac{1}{\la_k}\ld w,u_k\rd\right)\\
		 		   	\geq & \limsup_{k}  \left(\frac{1}{2 }\I\Om|\na u_k|^2dx+ \frac{1}{2}\iint_{\mb R^{2n}} |\mc B(u_k)(x,y)|^2d\nu-\I\Om \frac{H_1(x,\la_ku_k)}{\la_k^2}dx\right.\\
		 		   	 &+\left.\I\Om e_1u_kdx-\frac{1}{\la_k}\ld w,u_k\rd \right).
		 	\end{align*}
		 	Since $\|u_k\|=r$, up to a subsequence, $u_k \rightharpoonup u$ weakly in $H_0^1(\Om)$ with $\|u\| \leq r$. Since by \cite[Lemma 3.3]{C}, we have
		 	\begin{align}\label{eLJ_1}
		 		\lim_k \frac{H_1(x,\la_k u_k)}{\la_k^2} =\frac{\al}{2} u^2~\text{strongly in}~L^1(\Om),
		 		\end{align}
		 		we deduce that $u \ne 0$ and
		 		\begin{align}\label{e5.12}
		 			\frac{1}{2}\I\Om|\na u|^2dx + \frac{1}{2}\iint_{\mb R^{2n}} |\mc B(u)(x,y)|^2d\nu-\frac{\al}{2}\I\Om u^2dx+\I\Om e_1udx \leq 0.
		 		\end{align}
		 	On the other hand, since $\tilde{\Psi}_{\la_k}(u_k)<+\infty$, from the definition of $J_0$ it follows that $\la_ku_k >-u_0$ in $\Om$. Therefore $u\geq 0$ in $\Om$ and \eqref{e5.12} is equivalent to $\tilde{\Psi}_\infty \leq 0$, which is a contradiction to \eqref{e5.10}.\\
		 	$(b)$ Let $v \in H_0^1(\Om)\cap L^\infty_c(\Om)$ with $v \geq 0$. Then
		 	\begin{align*}
		 		\Psi_\la(tv)=&t^2\left(\frac{1}{2}\I\Om |\na v|^2dx +\frac12\iint_{\mb R^{2n}}|\mc B(v)(x,y)|^2d\nu+\frac{1}{t^2}\I\Om J_0(x,tv)dx-\frac{1}{t^2}\I\Om H_1(x,tv)dx\right.\\
		 		&\left.+\frac{\la}{t}\I\Om e_1 vdx- \frac{1}{t}\ld w,v\rd\right).
		 	\end{align*}
		 	Since $\I\Om |\na e_1|^2dx + \iint_{\mb R^{2n}}|\mc B(e_1)(x,y)|^2 d\nu \leq \al \I\Om e_1^2 dx$, by an approximation argument, we can take $v \in H_0^1(\Om) \cap L^\infty_c(\Om)$ with $v \geq 0$ such that
		 	\begin{align*}
		 		\I\Om |\na v|^2dx + \iint_{\mb R^{2n}}|\mc B(v)(x,y)|^2 d\nu \leq \al \I\Om v^2 dx.
		 		\end{align*}
		 	Choose $v$ as above and take into account the fact that $u_0 \geq C_1>0$ in the support of $v$ and arguing as in \eqref{eLJ_1}, we get
		 	\begin{align*}
		 		\lim_{t \ra +\infty} \Psi_\la(tv)\sim \lim_{t \ra +\infty}t^2\left(	\I\Om |\na v|^2dx + \iint_{\mb R^{2n}}|\mc B(v)(x,y)|^2 d\nu - \al \I\Om v^2 dx\right)=-\infty.
		 	\end{align*}
		 	This concludes the proof.
				\end{proof}
\subsubsection{Proof of existence result for the singular variational inequality}
				\textbf{Proof of Theorem \ref{t5.4}:}
					Let $\overline{\la}, r >0$ be as in assertion $(a)$ of Lemma \ref{t5.3} and take $\la >\overline{\la}$. Since $\Psi_\la(0)=0$, using Lemmas \ref{t5.2} and \ref{t5.3} it follows that $\Psi_\la$ satisfies the assumptions of Theorem \ref{TMPT}. Then Theorem \ref{TMPT} gives a critical point for $\Psi_\la$, say $u_1$ with $\Psi_\la(u_1)>0$.\\
					On the other hand, $\Psi_\la$ is weakly lower semicontinuous. Therefore {$\Psi_\la$} admits a minimum $u_2$ on the closed convex set $\{u \in H_0^1(\Om):\|u\| \leq r\}$ with $\Psi_\la(u_2) \leq 0$. Since $\|u_2\| \leq r$, $u_2$ is a local minimum of $\Psi_\la$ and hence another critical point of $\Psi_\la$.\\
                    Finally using Proposition \ref{p3.1} and Theorem \ref{t3.2}, we conclude that $u_0+u_1$ and $u_0 +u_2$ are two distinct solutions of \eqref{e5.14} in $H^1_{\text{loc}}(\Om)\cap L^1(\Om)$. \qed
\subsection{Proof of the main results} In this subsection we complete the proofs of Theorems \ref{t1.4} and \ref{t1.5}.\\
\textbf{Proof of Theorem \ref{t1.4}.} Let $w$ be of the form $(\mc H_2)$. Since $(h_1)$ implies $(h_1^\prime)$, we can apply Theorem \ref{t5.4}, obtaining two distinct solutions {$u_0+u_1, u_0+u_2 \in H^1_{\text{loc}}(\Om)\cap L^1(\Om)$} of \eqref{e5.14}. Now we need to pass from the variational inequality {\eqref{e5.14} to the equation $(\mc P_{\gamma,w})$}. In view of Theorem \ref{t1.1}, we only need $h(x,{u_0}+u_i)-\la e_1 \in L^1_{\text{loc}}(\Om)$ for $i=1,2$ which is in fact the case because of the assumption $(h_1)$ and the fact that $u_i \in H^1_{\text{loc}}(\Om)$.\\
Next we discuss the regularity of the solutions. For this let $u\in H^1_{\text{loc}}(\Om)\cap L^1(\Om)$ be any weak solution of $(\mc P_{\gamma,w})$}. Our first step is to show that $u \in L^\infty(\Om)$. To this end, we will use the following inequality (see e.g., \cite[page 879]{BCSS} for the fractional Laplacian 
\begin{align}\label{efli}
	(-\De)^s \psi(u) \leq \psi^\prime(u)(-\De)^s u,
\end{align}
where $\psi$ is a convex piecewise $C^1$ with bounded derivative function. Now let $\Upsilon:\mb R \ra \mb [0,1]$ be a $C^\infty(\mb R)$ convex increasing function such that $\Upsilon^\prime(t) \leq 1$ for all $t \in [0,1]$ and $\Upsilon^\prime(t)=1$ when $t \geq 1$. Define $\Upsilon_\epsilon(t)= \e \Upsilon(\frac{t}{\e})$. Then using the fact that $\Upsilon_\e$ is smooth, we obtain $\Upsilon_\e \ra (t-1)^+$ uniformly as $\e \ra 0$. This fact along with \eqref{efli} implies for any $\varphi \in C_c^\infty(\Om)$ with $\varphi \geq 0$ that
\begin{align*}
	\I\Om&\na \Upsilon_\e(u)\na \varphi dx +\iint_{\mb R^{2n}}\mc B(\Upsilon_\e(u))(x,y)\mc B(\varphi)(x,y)d\nu \\
	\leq& \Upsilon_\e^\prime(u)\left(	\I\Om\na u\na \varphi dx +\iint_{\mb R^{2n}}\mc B(u)(x,y)\mc B(\varphi)(x,y)d\nu\right)\\
	=&  \Upsilon_\e^\prime(u)\left( \I\Om u^{-\gamma}\varphi dx + \I\Om (h(x,u)-\la e_1)\varphi dx\right)	\\
    \leq & \chi_{\{u>1\}}\left( \I\Om u^{-\gamma}\varphi dx + \I\Om (|h(x,u)|+\la e_1)\varphi dx\right).
\end{align*}
Hence, as $\e \ra 0$ using $(h_1)$ we deduce that
\begin{align*}
	&\I\Om\na(u-1)^+\na \varphi dx +\iint_{\mb R^{2n}}\mc B((u-1)^+)(x,y)\mc B(\varphi)(x,y)d\nu\\
	&\leq \chi_{\{u>1\}}\left( \I\Om u^{-\gamma}\varphi dx + \I\Om (|h(x,u)|+\la e_1)\varphi dx\right)\leq C\int_{\Omega}(1+|(u-1)^+|^{q-1})\varphi\,dx,
\end{align*}
for any $q \in [2,2^\ast]$. Now using \cite[Theorem 1.1]{SVWZ} we have that $u \in L^\infty(\Om)$ (the only difference in the proof of \cite[Lemma 3.2]{SVWZ} is $``\leq"$ instead of $``="$ in the equation (3.3) there). This implies using $(h_1)$, that $w =h(x,u)-\la e_1 \in L^\infty(\Om)$ and so there exists $M_w, m_w>0$ such that $m_w u_0$ is a weak subsolution and $M_w u_0$ is a weak supersolution of $(\mc P_{\gamma,w})$. Hence by Lemma \ref{tcp} we have $m_w u_0 \leq u \leq M_w u_0$ which further implies that $u \in C(\overline{\Om})$ as $u_0 \in C(\overline{\Om})$ with $u_0=0$ in $\mb R^n \setminus \Om$ (see Proposition \ref{p3.1}) and $u^{-\gamma} \in L^\infty_{\text{loc}}(\Om)$. Finally as in the proof of \cite[Theorem 1.5, page 16]{AGS} we conclude that $u$ is in $C(\overline{\Om})\cap\left(\cap_{1\leq p<\infty}W^{2,p}_{\text{loc}}(\Om)\right)$ if $s \in (1,1/2]$ and in $C(\overline{\Om})\cap\left(\cap_{1\leq p<n/(2s-1)}W^{2,p}_{\text{loc}}(\Om)\right)$ if $s \in (1/2,1)$. \qed\\
\textbf{Proof of theorem \ref{t1.5}:} Let $w$ be of the form $(\mc H_2)$. Suppose on the contrary that there exists a sequence $\{\la_k\}_{k \in \mb N}$ and $\{u_k\}_{k\in\mb N}$ such that $\la_k \ra -\infty$ and $u_k\in H^1_{\text{loc}}(\Om)\cap L^1(\Om)$ is a weak solution of $(\mc P_{\gamma,w})$ with $\la=\la_k$. Without loss of generality, we may assume that $\la_k <0$. Also from Theorems \ref{t3.2} and \ref{t1.1}, we have $u_k-u_0 \in H_0^1(\Om)$. \\
\textbf{Case 1:} First suppose that $z_k:=(u_0-u_k)/\la_k$ is bounded in $H_0^1(\Om)$, and hence up to a subsequence $z_k \rightharpoonup \hat z$ weakly in $H_0^1(\Om)$. Moreover, we remark that $\hat z \geq 0$ in $\Om$. Indeed, taking into account the assumptions $(h_1),\,(h_2)$, it can be shown that up to a subsequence each $u_k$ is a weak supersolution of $(\mc P_{\gamma,0})$ and applying Lemma \ref{tcp}, it follows that $u_k\geq u_0$ in $\Om$, which further gives $\hat z\geq 0$ in $\Om$. Then we have for every $v \in H_0^1(\Om)\cap L^\infty_c(\Om)$ with $v \geq 0$ in $\Om$
\begin{align}\label{e5.15} \nonumber
	\I\Om \na z_k \na vdx+&\iint_{\mb R^{2n}}\mc B(z_k)(x,y)\mc B(v)(x,y)d\nu\\ =&-\frac{1}{\la_k}\I\Om\left((u_0-\la_kz_k)^{-\gamma}-u_0^{-\gamma}\right)v dx + \I\Om\left(\frac{h_1(x,-\la_kz_k)}{-\la_k}+e_1\right)vdx.
\end{align}
		Since $u_0\geq C>0$ on the support of $v$, we can pass to the limit as $k \ra \infty$  in \eqref{e5.15} and taking in to account \eqref{e5.8}, we obtain
		\begin{align*}
			\I\Om \na \hat z \na vdx+&\iint_{\mb R^{2n}}\mc B(z)(x,y)\mc B(v)(x,y)d\nu = \I\Om(\al \hat z +e_1)v dx,
		\end{align*}
		for every $v \in H_0^1(\Om)\cap L^\infty_c(\Om)$ with $v \geq 0$ in $\Om$. Using density arguments, we can choose $v =e_1$ above and obtain
		\begin{align}\label{5.16}
			\la_1\I\Om \hat ze_1dx = 	\I\Om \na \hat z \na e_1dx+\iint_{\mb R^{2n}}\mc B(\hat z)(x,y)\mc B(e_1)(x,y)d\nu = \al\I\Om \hat ze_1dx +\I\Om e_1^2 dx.
		\end{align}
		Now if $\hat z \equiv 0$, then \eqref{5.16} contradicts the fact that $e_1 \not \equiv 0$ in $\Om$. Further if $\hat z \not \equiv 0$, then using $\hat z \geq 0$, we get a contradiction to the assumption that $\al >\la_1$.\\
		\textbf{Case 2:} Now suppose that $\la_k/\|u_k-u_0\|$ is convergent to $0$. If we set $r_k =\|u_k-u_0\|$ and $z_k =(u_k-u_0)/r_k$, then $z_k \geq 0$ in $\Om$ and up to a subsequence $z_k \rightharpoonup \hat z$ weakly in $H_0^1(\Om)$ with $\hat z \geq 0$ in $\Om$ (which follows similarly as in Case 1). Now we claim that $\hat z \not \equiv 0$ in $\Om$. Indeed, for every $v \in H_0^1(\Om)\cap L^\infty_c(\Om)$ with $v \geq 0$ in $\Om$, we have
			\begin{align}\nonumber\label{e5.16}
			\I\Om \na z_k \na vdx+&\iint_{\mb R^{2n}}\mc B(z_k)(x,y)\mc B(v)(x,y)d\nu\\ =&\frac{1}{r_k}\I\Om\left((u_0+r_kz_k)^{-\gamma}-u_0^{-\gamma}\right)v dx + \I\Om\left(\frac{h_1(x,r_kz_k)}{r_k}-\frac{\la_k}{r_k}e_1\right)v\,dx.
		\end{align}
		By density choosing $v =z_k$ in \eqref{e5.16}, we obtain
		\begin{align*}
			1=&
            \frac{1}{r_k}\I\Om\left((u_0+r_kz_k)^{-\gamma}-u_0^{-\gamma}\right)z_k dx + \I\Om\left(\frac{h_1(x,r_kz_k)}{r_k}-\frac{\la_k}{r_k}e_1\right)z_kdx\\
			\leq& \I\Om\left(\frac{h_1(x,r_kz_k)}{r_k}-\frac{\la_k}{r_k}e_1\right)z_k dx.
		\end{align*}
		Taking again in to account \eqref{e5.9} and the fact that $\la_k/r_k \ra 0$ as $k \ra \infty$, it follows that
	\begin{align*}
	1 \leq {\al} \I\Om \hat z^2 dx,
	\end{align*}
	which implies $\hat z \not \equiv 0$ in $\Om$. Furthermore arguing as Lemma \ref{t5.2}, we get
		\begin{align*}
			\I\Om \na \hat z \na vdx+\iint_{\mb R^{2n}}\mc B(\hat z)(x,y)\mc B(v)(x,y)d\nu \geq \al\I\Om \hat z v dx,
		\end{align*}
		for every $v \in H_0^1(\Om)\cap L^\infty_c(\Om)$ with $v \geq 0$ in $\Om$. Now, using density, we take $v =e_1$ in above inequality, and using the facts that $\hat z \geq 0$ and $\hat z \not \equiv 0$, we obtain a contradiction to the assumption $\al > \la_1$. This completes the proof. \qed

		\section{Appendix}
		\subsection{Variational characterization}
			In this subsection, we present two essential results related to variational characterization (Theorems \ref{t3.2} and \ref{t1.1}), which played a key role in proving our main results. The following result establishes a connection between the solutions of the variational inequality and the minimizer of an appropriate functional.
		\begin{Theorem}\label{t3.2}
			Let $w \in H^{-1}(\Om)$ and $u \in H^1_{\text{loc}}(\Om) \cap L^1(\Om)$. Suppose $\Phi_w$ is as defined in \eqref{e3.11}. Then the following are equivalent:
			\begin{enumerate}
				\item [(a)] $u$ is the minimum of $\Phi_w$.
				\item [(b)] $u$ satisfies the following
				\begin{align}\label{e3.12}
					\begin{cases}
						u>0~\text{in}~\Om~\text{and}~u^{-\gamma} \in L^1_{\text{loc}}(\Om),\\
						\I\Om \na u\na(v-u)dx +\I{\mb R^n}\I{ \mb R^n}\mathcal{B}(u)(x,y)\mathcal{B}(v-u)(x,y)\,d\nu	-\I\Om u^{-\gamma}(v-u)dx\\
						~~~~~~~~~~~~~~~\geq \ld w,v-u\rd, \quad \forall v \in u+(H_0^1(\Om)\cap L^\infty_c(\Om))~\text{with}~v \geq 0~{in}~\Om,\\
						u \leq 0~\text{on}~\partial \Om.
					\end{cases}
				\end{align}
			\end{enumerate}
			In particular, for every $w \in H_0^1(\Om)$, problem \eqref{e3.12} has one and only one solution $u \in  H^1_{\text{loc}}(\Om) \cap L^1(\Om)$.
		\end{Theorem}
		\begin{proof}
			We assume that $(a)$ holds, that is $u$ is the minimum of $\Phi_w$. Since $\Phi_w$ is strictly convex, lower semicontinuous, and coercive, standard minimization techniques guarantee that $u$ is unique and $u \in u_0 + H_0^1(\Om)$ (and by Proposition \ref{p3.1}, $u \in L^1(\Om)$). Clearly $u$ lies in domain of $\Phi_w$, which implies $J _0(x,u-u_0) \in L^1(\Om)$ and using \eqref{edP} and \eqref{edj_0}, we conclude that
			\begin{align}\label{e3.13}
				u \geq 0~\text{in}~\Om.
			\end{align}
			Now let $v \in u_0 + H_0^1(\Om)$ be such that $J_0(x,v-u_0) \in L^1(\Om)$. Then $v \geq 0$ in $\Om$, and additionally, $v-u \in H_0^1(\Om)$. Since $D_t J_0(\cdot,t)$ is increasing (see \eqref{e3.10}), for $\hat z \in (\min\{(v-u_0),(u-u_0)\}, \max\{(v-u_0),(u-u_0)\})$, we deduce that
			\begin{align*}
				J_0(x,v-u_0) -J_0(x,u-u_0) =(u_0^{-\gamma}-(u_0+\hat z)^{-\gamma})(v-u) \geq (u_0^{-\gamma}-u^{-\gamma})(v-u),
			\end{align*}
			which in combination with $J_0(x,u-u_0) \in L^1(\Om)$ and $J_0(x,v-u_0) \in L^1(\Om)$ implies that
			\begin{align}\label{e3.14}
				\left({u_0^{-\gamma}} -{u^{-\gamma}}\right)(v-u) \in L^1(\Om).
			\end{align}	
			In particular we have
			\begin{align*}
				\left({u_0^{-\gamma}} -{u^{-\gamma}}\right) v \in L^1(\Om),~\forall v \in C_c^\infty(\Om)~\text{with}~v \geq 0
			\end{align*}
			and so $u^{-\gamma} \in L^1_{\text{loc}}(\Om)$ and $u>0$ in $\Om$.
			Now using convexity of $J_0(x,\cdot)$ we see that
			\begin{align*}
				J_0(x,u-u_0+t(v-u))=J_0(x,t(v-u_0)+(1-t)(u-u_0)) \in L^1(\Om),~ \forall t \in [0,1].
			\end{align*}
			Since $u$ is the point of minimum of $\Phi_w$, for $t\in [0,1]$ we get
			\begin{align}\nonumber \label{e3.16}
				0 \leq& \frac{\Phi_w(u+t(v-u))-\Phi_w(u)}{t}\\ \nonumber
				=&\I{\Om} \na(u-u_0)\na(v-u)dx +\iint_{\mb R^{2n}}\mc B (u-u_0)(x,y)\mc B(v-u)(x,y)d\nu -\ld w,v-u\rd\\ \nonumber
				&+ \frac{t}{2}\left(\| \na(v-u)\|_{L^2(\Om)}^2 +[v-u]^{2}_{s,\mb R^n}\right) \\
                &\qquad+ \frac{1}{t}\left(\I\Om J_0(x,u-u_0+t(v-u))dx-\I\Om J_0(x,u-u_0)dx\right)\\ \nonumber
				=&\I{\Om} \na(u-u_0)\na(v-u)dx +\iint_{\mb R^{2n}}\mc B (u-u_0)(x,y)\mc B(v-u)(x,y)d\nu-\ld w,v-u\rd\\ 
				&+ \frac{t}{2}\left(\| \na(v-u)\|_{L^2(\Om)}^2 +[v-u]^{2}_{s,\mb R^n}\right) + \I\Om\left({u_0^{-\gamma}}-{(u_0+z_t)^{-\gamma}}\right)(v-u)dx,
			\end{align}
			where $z_t \in (\min\{u-u_0+t(v-u),u-u_0\},\max\{u-u_0+t(v-u),u-u_0\})$. Recalling \eqref{e3.14} and that $v-u \in H_0^1(\Om)$, passing to the limit as $t \ra 0^+$ in \eqref{e3.16} we obtain
			\begin{align}\nonumber\label{e3.17}
				\I{\Om} \na(u-u_0)\na(v-u)dx +& \iint_{\mb R^{2n}}\mc B (u-u_0)(x,y)\mc B(v-u)(x,y)d\nu\\
				\geq& \I\Om\left({u^{-\gamma}}-{u_0^{-\gamma}}\right)(v-u)dx + \ld w,v-u\rd,
			\end{align}
			for every $v \in u_0 +H_0^1(\Om)$ such that $J_0(x,v-u_0) \in L^1(\Om)$. For $\e ,\mu >0$, let us define
			\begin{align*}
				Z= \min\{u-u_0,\e-(u_0-\mu)^+\}.
			\end{align*}
			Since $u_0 \in C(\overline{\Om})$, we have $Z \in H_0^1(\Om)$. Also either $Z =u-u_0$ or $\e =Z \leq u-u_0$ or $Z =\e+\mu-u_0$ and $u_0 \geq \mu$. In all three cases we have that $J_0(x,Z) \in L^1(\Om)$ and that
			\begin{align*}
				((u_0-\mu)^+ +u-u_0-\e)^+=u-u_0-Z \in H_0^1(\Om),
			\end{align*} 
			using \eqref{e3.14}
			\begin{align}\label{e3.18}
				\left({u_0^{-\gamma}} -{u^{-\gamma}}\right)(Z+u_0-u) \in L^1(\Om)
			\end{align}
			and using \eqref{e3.17}
			\begin{align}\label{e3.19}\nonumber 
				\I{\Om} \na(u-u_0)\na(Z+u_0-u)dx +&\iint_{\mb R^{2n}}\mc B (u-u_0)(x,y)\mc B(Z+u_0-u)(x,y)d\nu\\
				\geq& \I\Om\left({u^{-\gamma}}-{u_0^{-\gamma}}\right)(Z+u_0-u)dx +\ld w,Z+u_0-u\rd.
			\end{align}
			In particular, since $u\ne u_0 +Z$ implies $ u >\e$, from \eqref{e3.18} we have that both
			\begin{align} \label{e3.20} 
				{u_0^{-\gamma}} (Z+u_0-u) \in L^1(\Om)~\text{and}~{u^{-\gamma}}(Z+u_0-u) \in L^1(\Om).
			\end{align}
			Now using Proposition \ref{p3.1}, we have
			\begin{align}\label{e3.21}
				\I\Om \na u_0 \na \varphi +\iint_{\mb R^{2n}}\mc B(u_0)(x,y)\mc B(\varphi)(x,y)d\nu = \I\Om {\varphi}{u_0^{-\gamma}}dx,~ \forall \varphi \in C^\infty_c(\Om).
			\end{align}
			Using the local and nonlocal Kato inequalities (see \cite[Theorem 2.4]{V}and \cite{CV} respectively) we have
			\begin{align}\label {e3.22}
				\I\Om \na (u_0-\mu)^+ \na \varphi dx+\iint_{\mb R^{2n}}\mc B((u_0-\mu)^+)(x,y)\mc B(\varphi)(x,y)d\nu \leq \I\Om {\varphi}{u_0^{-\gamma}}dx,
			\end{align}
			for all $\varphi \in C^\infty_c(\Om), \varphi \geq 0$. Using standard arguments, we see that the inequality \eqref{e3.22} holds true for non-negative $\varphi \in H_0^1(\Om)$ with compact support contained in $\Om$. By density, let $ \{\varphi_k\}_{k\in\mathbb{N}} \in C^\infty_c(\Om)$ such that $\varphi_k^+ \ra u-u_0-Z $ in $H_0^1(\Om)$. Let us define
			\begin{align}\label{e3.23}
				\hat{\varphi}_k:= \min\{u-u_0-Z, \varphi_k^+\}.
			\end{align}
			Again since $u_0 \in C(\overline{\Om})$, we have $(u_0-\mu)^+ \in H_0^1(\Om)$. Now testing \eqref{e3.22} with $\hat{\varphi}_k$ defined in \eqref{e3.23} and passing to the limit using \eqref{e3.20} and dominated convergence theorem, we obtain
			\begin{align}\label {e3.24}\nonumber
				\I\Om \na (u_0-\mu)^+ \na (u-u_0-Z) dx+&\iint_{\mb R^{2n}}\mc B((u_0-\mu)^+)(x,y)\mc B(u-u_0-Z)(x,y)d\nu\\ \leq& \I\Om{(u-u_0-Z)}{u_0^{-\gamma}}dx.
			\end{align}
			Note that for any function $v$ since (see \cite[page 4046]{CMST} after equation (3.34) there)
			\begin{align*}
			\iint_{\mb R^{2n}} \mc B(v-v^+)(x,y)\mc B(v^+)(x,y) d\nu \geq 0,
			\end{align*}
			we have 
			\begin{align}\label{e3.25}
				\iint_{\mb R^{2n}} \mc B(v)(x,y)\mc B(v^+)(x,y) d\nu \geq [v^+]^2_{s,\mb R^n}.
			\end{align}
			Combining \eqref{e3.24} with \eqref{e3.19} and using \eqref{e3.25} with $v=(u_0-\mu)^++u-u_0-\e$ and recalling $v^+ = u-u_0-Z$, we obtain
			\begin{align*}
				\I\Om& |\na(u-u_0-Z)|^2dx\\
				\leq&  \I\Om |\na(u-u_0-Z)|^2dx +[u-u_0-Z]^2_{s,\mb R^n} \leq  \I\Om \na((u_0-\mu)^++u-u_0-\e)\na(u-u_0-Z)dx\\
				& +	\iint_{\mb R^{2n}} \mc B((u_0-\mu)^++u-u_0-\e)(x,y)\mc B(u-u_0-Z)(x,y) d\nu\\
				\leq& \I\Om {u^{-\gamma}}(u-u_0-Z)dx +\ld w,(u-u_0-Z)\rd \leq \e^{-\gamma}\I\Om(u-u_0-Z)dx+\ld w,(u-u_0-Z)\rd.
			\end{align*}
			Hence for any $\e >0$, $((u_0-\mu)^++u-u_0-\e)^+$ is uniformly bounded with respect to $\mu$ in $H_0^1(\Om)$. Using Fatou's lemma for $\mu \ra 0^+$, we have that $(u-\e)^+ \in H_0^1(\Om)$. This proves $u \leq 0$ on $\partial \Om$.\\
			Now let $v \in u+(H_0^1(\Om) \cap L_c^\infty(\Om))$ with $v \geq 0$ in $\Om$ and $\psi \in C_c^\infty(\Om)$, $\psi \geq 0$ in $\Om$ such that $\psi \equiv 1$ where $v\ne u$. Then for any $\e >0$, $J_0(x,v+\e \psi-u_0) \in L^1(\Om)$ and therefore by using \eqref{e3.17} we have
			\begin{align}\nonumber\label{e3.26}
				\I{\Om} \na(u-u_0)\na(v+\e\psi-u)dx +&\iint_{\mb R^{2n}}\mc B (u-u_0)(x,y)\mc B(v+\e\psi-u)(x,y)d\nu\\
				\geq& \I\Om\left({u^{-\gamma}}-{u_0^{-\gamma}}\right)(v+\e\psi-u)dx + \ld w,v+\e \psi-u\rd.
			\end{align}
			Taking $\e \ra 0$ in \eqref{e3.26} we obtain
			\begin{align}\nonumber\label{e3.27}
				\I{\Om} \na(u-u_0)\na(v-u)dx +&\iint_{\mb R^{2n}}\mc B (u-u_0)(x,y)\mc B(v-u)(x,y)d\nu\\
				\geq& \I\Om\left({u^{-\gamma}}-{u_0^{-\gamma}}\right)(v-u)dx + \ld w,v-u\rd.
			\end{align}
			By \eqref{e3.21} we also have that
			\begin{align*}
				\I\Om \na u_0 \na (v-u) +\iint_{\mb R^{2n}}\mc B(u_0)(x,y)\mc B(v-u)(x,y)d\nu = \I\Om {(v-u)}{u_0^{-\gamma}}dx
			\end{align*}
			and together with \eqref{e3.27}, this completes the proof of  \eqref{e3.12}.\\
			Conversely, let $(b)$ holds, that means $u$ is a solution to \eqref{e3.12} and let $\tilde{u} \in H_{\mathrm{loc}}^1(\Om)\cap L^1(\Om)$ be the minimum of the functional $\Phi_w$. Then, as we just proved above, $\tilde{u}$ satisfies \eqref{e3.12}. Thus both $u$ and $\tilde{u}$ are weak sub-supersolution to the problem $(\mc P_{\gamma,w})$. Hence by Lemma \ref{tcp}, we have $u =\tilde{u}$ i.e., $u$ is the minimum of $\Phi_w$.
		\end{proof}

The following result offers a variational characterization of weak solutions to the mixed local-nonlocal singular problem $(\mc P_{\gamma,w})$ for any $\gamma>0$ and $w\in H^{-1}(\Om)$. 
	\begin{Theorem}\label{t1.1}
		Let $\gamma >0$ and $u \in H^1_{\text{loc}}(\Om)\cap L^1(\Om)$. Consider the following two problems $(\mc G)$ and $(\mc H)$:
		\begin{align*}
			(\mc G)	\begin{cases}
				u>0~\text{in}~\Om~\text{and}~u^{-\gamma} \in L^1_{\text{loc}}(\Om),\\
				\I\Om \na u \na \varphi + \iint_{\mb R^{2n}} \mathcal{B}(u)(x,y)\mathcal{B}(\varphi)(x,y)\,d\nu-\I\Om u^{-\gamma}\varphi dx =\ld w,\varphi\rd, \forall \varphi \in C_c^\infty(\Om),\\
				u \leq 0~\text{on}~\partial \Om
			\end{cases}
		\end{align*}
		and
		\begin{align*}
			(\mc H)	\begin{cases}
				u>0~\text{in}~\Om~\text{and}~u^{-\gamma} \in L^1_{\text{loc}}(\Om),\\
				\I\Om \na u \na (v-u)dx + \iint_{\mb R^{2n}} \mc Bu(x,y) \mc B (v-u)(x,y)d\nu-\I\Om u^{-\gamma}(v-u) dx \geq \ld w,v-u\rd,\\
				~~~~~~~~~~~~~~~~~~~~~~~~~\forall v \in u+ (H_0^1(\Om) \cap L_c^\infty(\Om))~\text{with}~v\geq 0~\text{in}~\Om,\\
				u \leq 0~\text{on}~\partial \Om,
			\end{cases}
		\end{align*}
        If $w \in H^{-1}(\Om)$, then $(\mc G)$ implies $(\mc H)$. Moreover, if $w \in L^{1}_{\mathrm{loc}}(\Om)$, then $(\mc H)$ implies $(\mc G)$.
	\end{Theorem}
\begin{proof}
Let $w\in H^{-1}(\Om)$. If $u$ satisfies $(\mc G)$, using a density argument it follows that
		\begin{align*}
			\I\Om \na u \na \varphi + \iint_{\mb R^{2n}} \mc Bu(x,y) \mc B \varphi(x,y)d\nu-\I\Om u^{-\gamma}\varphi dx =\ld w,\varphi\rd, \forall \varphi \in H_0^1(\Om)\cap C_c^\infty(\Om),
		\end{align*}
		whence $u$ satisfies $(\mc H)$.\\
		Now suppose that $w \in L^1_{\text{loc}}(\Om)$ and that $u$ satisfies $(\mc H)$. It is readily seen that, for every $\varphi \in C_c^\infty(\Om)$ with $\varphi \geq 0$,
		\begin{align}\label{e3.28}
			\I\Om \na u \na \varphi + \iint_{\mb R^{2n}} \mc Bu(x,y) \mc B \varphi(x,y)d\nu-\I\Om u^{-\gamma}\varphi dx \geq\int_{\Om}w\varphi\,dx.
		\end{align} 
		Now suppose $ \varphi \in C_c^\infty(\Om)$ with $\varphi \leq 0$. For $t>0$, let us define $\varphi_t =(u+t\varphi)^+$. Let us denote $K_{\varphi_t}=\text{supp}(\varphi_t)$ and $K_{\varphi_t}^c=\mb R^n \setminus K_{\varphi_t}$. Setting 
		\begin{align}\label{e3.29}
			v_t =\frac{(\varphi_t-u)}{t},
		\end{align}
		we have
		\begin{align}\label{e3.30}
			\I{K_{\varphi_t}} \na u \na \varphi dx \geq -\frac{1}{t}\I{K_{\varphi_t}^c}|\na u|^2dx +\I{K_{\varphi_t}} \na u \na \varphi dx \geq \I\Om \na u \na v_tdx.
		\end{align}
		Also 
		\begin{align}\nonumber\label{e3.31}
			&\iint_{\mb R^{2n}}\mathcal{B}(u)(x,y)\mathcal{B}(v_t)(x,y)\,d\nu\\ \nonumber
			=& \iint_{\mb R^{2n}\setminus(K_{\varphi_t}^c\times K_{\varphi_t}^c)}\mathcal{B}(u)(x,y)\mathcal{B}(v_t)(x,y)\,d\nu -\frac{1}{t}\iint_{K_{\varphi_t}^c\times K_{\varphi_t}^c} {|u(x)-u(y)|^2}\,d\nu\\ \nonumber
			\leq & \iint_{\mb R^{2n}\setminus(K_{\varphi_t}^c\times K_{\varphi_t}^c)} \mathcal{B}(u)(x,y)\mathcal{B}(v_t)(x,y)\,d\nu\\ \nonumber
			=& \iint_{K_{\varphi_t}\times K_{\varphi_t}} \mathcal{B}(u)(x,y)\mathcal{B}(\varphi)(x,y)\,d\nu
			+2 \iint_{(K_{\varphi_t}\times K_{\varphi_t}^c)\cap\{u(x)\geq u(y)\}} \mathcal{B}(u)(x,y)\mathcal{B}(v_t)(x,y)\,d\nu\\ 
			+& 2 \iint_{(K_{\varphi_t}\times K_{\varphi_t}^c)\cap\{u(x)<u(y)\}} \mathcal{B}(u)(x,y)\mathcal{B}(v_t)(x,y)\,d\nu:= I_1 +2I_2+2I_3. 
		\end{align}
		Now we estimate $I_1,~I_2$ and $I_3$. For this, first note that since $(\mc H)$ holds, by Theorem \ref{t3.2} we have $u \in H^1_{\text{loc}}(\Om){\cap L^1(\Om)}$ is the minimum of $\Phi_w$. This implies that $u\in u_0 +H^1_0(\Om)$.
        Thus in view of Proposition \ref{p2.6} and since $\varphi \in C^\infty_c(\Om)$, we obtain 
		\begin{equation}\label{e3.32}
			\begin{split}
				I_1 = \iint_{K_{\varphi_t}\times K_{\varphi_t}} \mathcal{B}(u)(x,y)\mathcal{B}(\varphi)(x,y)\,d\nu
				\leq \iint_{\Om\times\Om} |\mathcal{B}(u)(x,y)|\,|\mathcal{B}(\varphi)(x,y)|\,d\nu <+\infty. 
			\end{split}
		\end{equation}
		This means that
		\begin{equation}\label{3.33}
			\begin{split}
				\frac{\mathcal{B}(u)(x,y)\mathcal{B}(\varphi)(x,y)}{|x-y|^{n+2s}}\cdot \chi_{K_{\varphi_t}\times K_{\varphi_t}}(x,y)
				&\leq \frac{|\mathcal{B}(u)(x,y)\mathcal{B}(\varphi)(x,y)|}{|x-y|^{n+2s}} \in L^1(\Om \times \Om),
			\end{split}
		\end{equation}
		where by $\chi_V$ we denote the characteristic function of a set $V$. Using the definition of $v_t$ (see \eqref{e3.29}), we obtain
		\begin{align}\label{e3.34} \nonumber
			I_2 =& \iint_{(K_{\varphi_t}\times K_{\varphi_t}^c)\cap\{u(x)\geq u(y)\}} {\mathcal{B}(u)(x,y))(\varphi(x)-u(y)/t)}\,d\nu\\ \nonumber
			\leq &\iint_{(K_{\varphi_t}\times K_{\varphi_t}^c)\cap\{u(x)\geq u(y)\}} \mathcal{B}(u)(x,y)\mathcal{B}(\varphi)(x,y)\,d\nu\\ 
			\leq & \iint_{\Om \times \Om} |\mathcal{B}(u)(x,y)\mathcal{B}(\varphi)(x,y)|\,d\nu+\iint_{\Om \times (\mb R^n \setminus \Om)}|\mathcal{B}(u)(x,y)\mathcal{B}(\varphi)(x,y)|\,d\nu.
		\end{align}
		Now the first integral on R.H.S. of \eqref{e3.34} is finite, see \eqref{e3.32}. Also noting that $u(x)=\varphi(x) =0$ on $\mb R^n \setminus \Om$, $\text{dist}(\partial K_{\varphi},\partial \Om)= \hat{r}$ (say) and $u \in L^1(\Om)$, we conclude that
		\begin{align}\nonumber
			\iint_{\Om \times (\mb R^n \setminus \Om)}|\mathcal{B}(u)(x,y)\mathcal{B}(\varphi)(x,y)|\,d\nu\nonumber &\leq C(s,n,\Om)\I\Om |u(x)||\varphi(x)|dx\I{|y|\geq \hat{r}} \frac{1}{|y|^{n+2s}}dy < +\infty.
		\end{align}
		Hence, from \eqref{e3.34} we deduce that
		\begin{equation}\label{e3.35}
			\begin{split}
				\frac{\mathcal{B}(u)(x,y)\mathcal{B}(v_t)(x,y)}{|x-y|^{n+2s}}\cdot \chi_{(K_{\varphi_t}\times K_{\varphi_t}^c)\cap\{u(x)\geq u(y)\}}\leq \frac{\mathcal{B}(u)(x,y)\mathcal{B}(\varphi)(x,y)}{|x-y|^{n+2s}} \in L^1(\Om\times (\mb R^n \setminus \Om)).
			\end{split}
		\end{equation}
		Again using the definition of $v_t$, we see that
		\begin{align}\nonumber\label{e3.36}
			I_3 =\iint_{(K_{\varphi_t}\times K_{\varphi_t}^c)\cap\{u(x)<u(y)\}}& \mathcal{B}(u)(x,y)(\varphi(x)+u(y)/t)\,d\nu\\
			\leq&  -\frac{1}{t} \iint_{(K_{\varphi_t}\times K_{\varphi_t}^c)\cap\{u(x)<u(y)\}} {|u(x)-u(y)|^2}\,d\nu \leq 0.
		\end{align}
		Using \eqref{e3.30}, \eqref{e3.31} and \eqref{e3.36}, we obtain
		\begin{align}\nonumber\label{3.35}
			\I\Om \na u \na v_tdx +& \iint_{\mb R^{2n}}\mathcal{B}(u)(x,y)\mathcal{B}(v_t)(x,y)\,d\nu \leq 	\iint_{K_{\varphi_t}\times K_{\varphi_t}} \mathcal{B}(u)(x,y)\mathcal{B}(\varphi)(x,y)\,d\nu \\
			&+\I{K_{\varphi_t}} \na u\na\varphi dx+2 \iint_{(K_{\varphi_t}\times K_{\varphi_t}^c)\cap\{u(x)\geq u(y)\}} \mathcal{B}(u)(x,y)\mathcal{B}(v_t)(x,y)\,d\nu.
		\end{align}
		Observe that $|v_t| \leq |\varphi|$. Since $(\mc H)$ holds, we conclude from \eqref{3.35} that
		\begin{align}\nonumber\label{e3.37}
			&\I{K_{\varphi_t}} \na u \na \varphi dx+ \iint_{K_{\varphi_t}\times K_{\varphi_t}} \mathcal{B}(u)(x,y)\mathcal{B}(\varphi)(x,y)\,d\nu\\ 
			+&2 \iint_{(K_{\varphi_t}\times K_{\varphi_t}^c)\cap\{u(x)\geq u(y)\}} \mathcal{B}(u)(x,y)\mathcal{B}(v_t)(x,y)\,d\nu
			\geq \I\Om u^{-\gamma}v_t dx + \I\Om w v_tdx.	\end{align}
		Using \eqref{3.33}, \eqref{e3.35} and recalling that $u>0$ in $\Om$, by the dominated convergence theorem from \eqref{e3.37}, we finally get
		\begin{align}\nonumber\label{e3.38}
			\I{\Om} \na u \na \varphi dx+& \iint_{{\Om}\times{\Om}} \mathcal{B}(u)(x,y)\mathcal{B}(\varphi)(x,y)\,d\nu \\
			+2& \iint_{\Om\times (\mb R^n \setminus \Om) } \mathcal{B}(u)(x,y)\mathcal{B}(\varphi)(x,y)\,d\nu 
			\geq  \I\Om u^{-\gamma}\varphi dx + \I\Om w \varphi dx.	\end{align}
		Up to change of variable in the third integral on L.H.S. of \eqref{e3.38} we deduce
		\begin{align} \label{e3.39}
			\I{\Om} \na u \na \varphi dx+ \iint_{{\mb R^{2n}n}} \mathcal{B}(u)(x,y)\mathcal{B}(\varphi)(x,y)\,d\nu  
			\geq  \I\Om u^{-\gamma}\varphi dx + \I\Om w \varphi dx,
		\end{align}
		for all $\varphi \in C_c^\infty(\Om)$ with $\varphi \leq 0$. Combining \eqref{e3.28} and \eqref{e3.39} we infer that $u$ satisfies $(\mc G)$. This completes the proof.
        \end{proof}

	\subsection{Decomposition result} The final main result is a decomposition theorem, which is crucial to prove symmetry result. This result will be a consequence of the variational characterization Theorem \ref{t1.1} stated above.
    To this end, let $f: \Om \times \mb R \ra \mb R$ be a Carath\'eodory function that satisfies the following growth assumption:
	\begin{align*}
		(F): |f(x,t)| \leq |r(x)| + a |t|^{\frac{n+2}{n-2}}~\text{for }~x \in \Om~\text{every}~t \in \mb R,~\text{where}~r\in L^{\frac{2n}{n-2}}(\Om)\text{ and }a \in \mb R,~a>0.
	\end{align*}
	Further, let $u_0$ be the unique weak solution of purely singular problem $(\mc P_{\gamma,0})$ given by Proposition \ref{p3.1}. Next we define $f_1(x,t) = f(x,u_0(x)+t)$, $F_1(x,t)=\I0^t{f_1 }(x,s)ds$ and the $C^1$ functional $ J: H_0^1(\Om) \ra \mb R$ by
	\begin{equation}\label{J}
		J(u) =- \I\Om F_1(x,u)dx.
	\end{equation}
	Finally, we define $\Psi: H_0^1(\Om) \ra (-\infty,+\infty]$ by
	\begin{align}\label{edpsi}
		\Psi(v) =J(v) + K(v),
	\end{align}
    where $K$ is defined in \eqref{eDK}. We have the following decomposition result:
	\begin{Theorem}\label{t1.2}
		For every $\gamma >0$, the following assertions are equivalent:
		\begin{itemize}
			\item [(a)] $u \in H^1_{\text{loc}}(\Om) \cap L^{\frac{2n}{n-2}}(\Om)$ satisfies weakly
			\begin{align} \label{emdp}
				\begin{cases}
					u>0~\text{in}~\Om~\text{and}~u^{-\gamma} \in L^1_{\text{loc}}(\Om),\\
					\mc M u =u^{-\gamma} + f(x,u) ~\text{in}~ \mc D^\prime(\Om),\\
					u \leq 0~\text{on}~\partial \Om,
				\end{cases}
			\end{align}
			where $f$ satisfies the hypothesis \textbf{$(F)$}.
			\item[(b)] $u \in u_0 +H_0^1(\Om)$ and $u-u_0$ is a critical point of $\Psi$ in the sense of Definition \ref{dcp}.
		\end{itemize}
	\end{Theorem}
		\begin{proof}
        Suppose $u$ satisfies $(a)$.
        Let $w=f(x,u)=f_1(x,u-u_0)$. Then $w \in H^{-1}(\Om)$. By Theorems \ref{t1.1} and \ref{t3.2} we have that $u =u_0 + H_0^1(\Om)$ and $u$ minimizes $\Phi_w$ defined in \eqref{e3.11}, i.e. for all $v \in H_0^1(\Om)$ we have
		\begin{align}\nonumber\label{e3.40}
			&\frac{1}{2 }\I\Om|\na v|^2dx + \frac{1}{2}\iint_{\mb R^{2n}}|\mc B(v)(x,y)|^2\,d\nu +\I\Om J_0(x,v)dx\\ \nonumber
			\geq& \frac{1}{2 }\I\Om|\na (u-u_0)|^2dx + \frac{1}{2}\iint_{\mb R^{2n}}|\mathcal{B}(u-u_0)(x,y)|^2\,d\nu +\I\Om J_0(x,u-u_0)dx\\
			&-\ld J^\prime(u-u_0),v-(u-u_0)\rd,
		\end{align}
		that is 
		\begin{align*}
			\ld J^\prime(u-u_0),v-(u-u_0))\rd + K(v)-K(u-u_0) \geq 0,
		\end{align*}
		where $K$ is defined by \eqref{eDK}.
		Recalling \eqref{edpsi}, $u-u_0$ is a critical point of $\Psi$ (see \eqref{edpsi}) in the sense of Definition \ref{dcp}. This proves that $u$ satisfies $(b)$.\\
		Conversely, assume that $u$ satisfies $(b)$. Then $u$ satisfies \eqref{e3.40} and using Proposition \ref{p3.1} we deduce that $u \in H^1_{\text{loc}}(\Om) \cap L^{\frac{2n}{n-2}}(\Om)$. Therefore $w =f(x,u)=f_1(x,u-u_0) \in H^{-1}(\Om)\cap L^1_{\text{loc}}(\Om)$. By Theorems \ref{t1.1} and \ref{t3.2} we conclude that $u$ is a weak solution to \eqref{emdp}. This concludes the proof.
        \end{proof}

	\section*{Acknowledgement} P. Garain gratefully acknowledges IISER Berhampur for the seed grant: IISERBPR/RD/OO/2024/15, Date: February 08, 2024. G. C. Anthal also appreciates the financial support provided by IISER Berhampur.

\end{document}